\documentclass{amsart}
\usepackage[utf8]{inputenc}
\usepackage{amsfonts}
\usepackage{amsmath,amssymb}
\usepackage{amsthm}
\usepackage{mathtools}
\usepackage[utf8]{inputenc}
\usepackage[english]{babel}
\usepackage[dvipsnames]{xcolor}
\usepackage{theoremref}
\usepackage[colorlinks=true]{hyperref}
\usepackage{enumitem}
\usepackage{graphicx}
\usepackage{stmaryrd}
\usepackage{tikz-cd}
\usepackage{comment}
\usepackage{microtype}

\usepackage{todonotes}

\newtheorem{theorem}{Theorem}[section]
\newtheorem{lemma}[theorem]{Lemma}
\newtheorem{corollary}[theorem]{Corollary}
\newtheorem{proposition}[theorem]{Proposition}
\newtheorem{definition}[theorem]{Definition}
\newtheorem{remark}[theorem]{Remark}

\newcommand{\bX}{\mathbb{X}}
\newcommand{\bx}{{\bf x}}
\newcommand{\by}{{\bf y}}

\newcommand{\bV}{\mathbb{V}}
\newcommand{\bY}{\mathbb{Y}}

\newcommand{\bP}{\mathbb{P}}

\newcommand{\bL}{\mathbb{L}}

\newcommand{\bB}{\mathbb B}

\newcommand{\Q}{\mathbb{Q}}
\newcommand{\Z}{\mathbb{Z}}

\newcommand{\h}{\mathbb{H}}

\newcommand{\Oo}{\mathcal{O}}

\newcommand{\cZ}{\mathcal{Z}}
\newcommand{\cC}{\mathcal{C}}
\newcommand{\cY}{\mathcal{Y}}

\newcommand{\cN}{\mathcal{N}}
\newcommand{\cL}{\mathcal{L}}

\newcommand{\cM}{\mathcal{M}}
\newcommand{\cF}{\mathcal{F}}
\newcommand{\cT}{\mathcal{T}}
\newcommand{\Exc}{\mathrm{Exc}}
\newcommand{\id}{\mathrm{id}}
\newcommand{\Kra}{\mathrm{Kra}}
\newcommand{\Nilp}{\mathrm{Nilp}}
\newcommand{\Int}{\mathrm{Int}}

\newcommand{\Spec}{\mathrm{Spec}\, }
\newcommand{\Spf}{\mathrm{Spf}\, }
\newcommand{\SpfOF}{{\mathrm{Spf}\,\mathcal{O}_{\breve{F}} }}

\newcommand{\Herm}{\mathrm{Herm}}
\newcommand{\Hom}{\mathrm{Hom}}
\newcommand{\End}{\mathrm{End}}

\newcommand{\q}{q}

\font\cute=cmitt10 at 12pt 
\newcommand{\kay}{{\text{\cute k}}}

\newcommand{\tr}{\operatorname{tr}}
\newcommand{\norm}{\operatorname{N}}
\newcommand{\ff}{\operatorname{ if }}

\newcommand{\kzxz}[4]{\left(\begin{smallmatrix} #1 & #2 \\ #3 & #4\end{smallmatrix}\right) }

\newcommand{\SL}{\operatorname{SL}}
\newcommand{\ord}{\operatorname{ord}}

\numberwithin{equation}{section}

\begin{document}

\title[The Kudla-Rapoport conjecture at a ramified prime for $U(1, 1)$]{The Kudla-Rapoport conjecture at a ramified prime for $U(1, 1)$}
\author{Qiao He, Yousheng Shi, and  Tonghai Yang}
\address{Department of Mathematics, University of Wisconsin Madison, Van Vleck Hall, Madison, WI 53706, USA}
\email{qhe36@wisc.edu}
\email{shi58@wisc.edu}
\email{thyang@math.wisc.edu}


\thanks{The first and the third authors are partially supported by a NSF grant DMS-1762289.}

\begin{abstract} In this paper, we proved a local arithmetic Siegel-Weil formula for a $U(1, 1)$-Shimura variety at a ramified prime, a.k.a. a Kudla-Rapoport conjecture at a ramified case. The formula needs to be modified from the original Kudla-Rapoport conjecture. In the process, we also gives an explicit decomposition of the special divisors of the Rapoport-Zink space of unitary type $(1, 1)$ (Kr\"amer model). A key ingredient is to relate the Rapoport-Zink space to the Drinfeld upper plane.
		
	\end{abstract}

\maketitle

    \setcounter{tocdepth}{1}
	\tableofcontents
	\section{Introduction}
	The  Arithmetic Siegel-Weil formula is part of the Kudla program, revealing a deep relation between the derivative of some Eisenstein series and the arithmetic degree or intersection of arithmetic special cycles on certain Shimura varieties of unitary or orthogonal type. It was started with Kudla's seminal paper in 1997 (\cite{Kudla97}), and the first two complete examples were worked out by Kudla, Rapoport, and one of the authors (T.Y.) (\cite{KRYtiny}, \cite{KRYcomp}, and \cite{KRYbook}). Afterwards, more progress have been made in other cases, e.g.,  \cite{HYBook},  \cite{Liu}, \cite{DY1}, \cite{LZ}, \cite{GS}, and \cite{BY20}.
	In almost all cases, the proof  is via  term by term comparison  on $T$-th Fourier coefficients. When  $T$ is non-singular, the $T$-th coefficient of the Eisenstein series factors through product of local Wittaker functions while the arithmetic intersection (degree) is often supported on a single prime. The key step is a local identity---local arithmetic Siegel-Weil formula: a deep relation between the derivative of the local Whittaker function and the `local intersection number' at a special point. In particular, the `local intersection index' is independent of the choice of `special points' (depending only on $T$)---an amazing phenomenon itself. At the infinite prime, it was proved by Yifeng Liu in the unitary case (\cite{Liu}), inspired by Kudla's original proof for Shimura curves(\cite{Kudla97}). The idea does not seem to extend to orthogonal cases.  Two different proofs in the general orthogonal case were given independently by Garcia and Sankaran \cite{GS} and Bruinier and Yang (\cite{BY20}). At a `good' finite  prime (everything is `unramified'),  Kudla and Rapoport formulated a precise local arithmetic Siegel-Weil formula, the so-called Kudla-Rapoport conjecture. This conjecture was recently proved by Li and Zhang (\cite{LZ}, see also \cite{BY20} for a special case) in the unitary case.  More recently, Cho (\cite{Cho}) considered the unitary $U(n, 1)$ case of miniscule parahoric level over an  unramified quadratic extension $E/F_0$ (of local $p$-adic fields),  and made some precise conjectural local arithmetic Siegel-Weil formula. The case $n=1$ was essentially  Sankaran's early result (\cite{San3}). In Shi's work (\cite{Shi2}) and  this paper, we consider the simplest case $U(1, 1)$ over a ramified quadratic local field extensions $F/F_0$. Ramification causes a lot of technical difficulty,  including singularity and so on. The resulting formula also needs modification  in general.  We hope to extend the result to general $U(n, 1)$ in a sequel. We now set up some notations and describe the main results in some detail.

Let $F=F_0(\pi)$ be  a ramified quadratic  field extension of a $p$-adic field $F_0$ with $ p>2$ with $\pi_0=\pi^2$ being a uniformizer of $F_0$. Let $\kay=\mathbb F_q$ be the residue field of $F_0$ (also $F$) with the algebraic closure $\bar\kay$. We denote the Galois conjugate of $a\in F$ by $\bar{a}$. Let $(V, (\, , \, ))$ be a Hermitian space over $F$ of dimension $2$  together with an $\Oo_F$-unimodular lattice $L$. Up to equivalence, we can describe $L$ explicitly as follows. Let $B$ be a quaternion algebra over $F_0$ together with an embedding $\Oo_F \subset \Oo_B$ (when $B=M_2(F_0)$, we take $\Oo_B=M_2(\Oo_{F_0})$). Choose $\delta \in  \Oo_B$ with $\delta x = \bar x  \delta$ for all $x \in \Oo_F$  and $\Delta =\delta^2 \in \Oo_{F_0}^\times$. Then we can take $L=\Oo_F + \Oo_F \delta$ with $z_i =x_i + y_i \delta$
$$
(z_1, z_2)= \tr_F(z_1 z_2^\iota) = x_1 \bar{x}_2 - y_1 \bar y_2 \Delta.
$$
Here $\tr_F(x+ y \delta) =x \in F$. Notice that $V=B$ in this case.

	Recall that a strict formal $\Oo_{F_0}$-module over an $\Oo_{F_0}$-scheme $S$ is a formal $p$-divisible group $X$ over $S$ with an action $\Oo_{F_0}\rightarrow \End(X)$ such that the induced $\Oo_{F_0}$-action on the Lie algebra $\mathrm{Lie}\, X$ is via the structural morphism $\Oo_{F_0}\rightarrow \Oo_S$.
	Let $\bY =(\bY, \iota_\bY, \lambda_\bY)$ be the unique supersingular strict formal $\Oo_{F_0}$-module of dimension 1 and $F_0$ height $2$ over $\bar\kay$ with an $\Oo_F$-action $\iota_\bY$ and a principal polarization $\lambda_\bY: \bY \rightarrow \bY^\vee=\bY$. Let $(\bX, \iota_\bX, \lambda_{\bX})$ be a  strict formal $\Oo_{F_0}$-module of dimension $2$ and $F_0$ height $4$ over $\bar\kay$, together with the $\Oo_F$-action $\iota_\bX$ and principal polarization $\lambda_\bX: \bX \mapsto \bX^\vee$. Let $\bL= \Hom_{\Oo_F}(\bY, \bX)$ be endowed with the Hermitian form
\begin{equation} \label{eq:V-Hermitian}
h(\bx,\by)=\lambda_{\bY}^{-1}\circ {\by}^\vee \circ   \lambda_{\bX}\circ \bx \in \End_{\Oo_F}(\bY)\xrightarrow[\sim]{\iota_{\bY}^{-1}} \Oo_F.
\end{equation}
The Hermitian lattice $(\bL, h)$ is $\Oo_F$-unimodular as the polarizations are principal. Let $\bV= \bL \otimes_{\Z_p} \Q_p$ be the associated Hermitian space.  To describe the Rapoport-Zink (RZ) space for the Shimura curve related to $L$,  we require
\begin{equation}
\det L /\det \bL \notin  \norm_{F/F_0} F^\times,
\end{equation}
i.e., $V$ and $\bV$ represent two different non-degenrate Hermitian spaces over $F$ of dimension $2$. By \cite[Remark 4.2]{RTW} (also see \cite[Proposition II.5.2]{BC} and the discussion below on the isomorphism between $\cN$ and $\breve{\Omega}$), the determinant condition determines $\bX$ uniquely up to $\Oo_{F}$-linear quasi-isogenies that preserve the polarization, hence determines the corresponding RZ space.

	Let $\breve{F}$ be the completion of the maximal unramified extension of $F$ with ring of integers $\Oo_{\breve{F}}$, and similarly for $\breve{F}_0$. We also denote $W=\Oo_{\breve{F}_0}$.
Let  $\mathrm{Nilp}_{\Oo_{\breve{F}}}$ be the category of $\Oo_{\breve{F}}$-schemes $S$ such that $\pi \cdot \Oo_{S}$ is a locally nilpotent ideal sheaf. Recall the RZ-space $\cN_{(1,1)}$ over $\Oo_{\breve{F}}$ which represents the following moduli problem ( \cite{RTW}): for $S\in \mathrm{Nilp}_{\Oo_{\breve{F}}}$, $\cN_{(1,1)}(S)$ is the groupoid of isomorphism classes of quadruples $(X,\iota,\lambda,\rho)$ given as follows:
	\begin{enumerate}
		\item
		$X$ is a strict formal $\Oo_{F_0}$-module over $S$ of dimension $2$ and $F_0$ height $4$;
		
		\item $\iota:\Oo_{F}\rightarrow \End (X)$ is an $\Oo_F$-action on $X$   satisfying Kottwitz condition:
		\[\mathrm{char}(\iota(\pi)|\mathrm{Lie}X)=(T-\pi)(T+\pi)=T^2 -\pi_0;\]
		
		\item
		$\lambda:X\rightarrow X^{\vee}$ is a principal quasi-polarization whose associated Rosati involution induces on $\Oo_F$ the nontrivial automorphism over $F_0$;
		
		\item Finally, $\rho:X\times_{S}\bar{S}\rightarrow \bX\times_{\Spec k}\bar{S}$ is a $\Oo_F$-linear quasi-isogeny of height 0 such that $\lambda$ and $\rho^{*}(\lambda_{\bX})$ differ locally on $\bar{S}$ by a factor in $\Oo_{F_0}^{\times}$.
	\end{enumerate}
	

  Although $\cN_{(1,1)}$ is normal and Cohen-Macaulay (see \cite{P}), it is not regular. There is a desingularization, called   the Kr\"amer model $\cN_{(1, 1)}^{\mathrm{Kra}}$(\cite{Kr}), which represents the isomorphism classes of quintuples $(X,\iota,\lambda,\rho,\mathcal{F}_X)$, where $(X,\iota,\lambda,\rho) \in \cN_{(1, 1)}(S)$,  and
\begin{enumerate}
\item
$\mathcal{F}_X$ is an $\Oo_F\otimes\Oo_S$-submodule of $\mathrm{Lie}(X)$ with $\Oo_S$-rank 1 and is a direct $\Oo_S$-summand of $\mathrm{Lie}(X)$;
\item  $\Oo_F$ acts on $\mathcal{F}_X$ via the structure morphism $\Oo_F\rightarrow \Oo_S$ and acts on $\mathrm{Lie}(X)/\mathcal{F}_X$ via the Galois conjugate of the structure morphism.
\end{enumerate}
By \cite[Proposition 2.7]{Shi2}, $\cN^\Kra$ is the blow up of $\cN_{(1,1)}$ along its singular locus which consists of ``superspecial" points.

Let $\cN_{(1, 0)}$ be the similar Rapoport-Zink space  with the framing object $\bX$ replaced by $\bY$. Then $\cN_{(1, 0)}$ is smooth and has a universal strict formal $\Oo_{F_0}$-module $\cY$ over $\Oo_{\breve{F}}$. For  every $S\in \mathrm{Nilp}_{\Oo_{\breve{F}}}$, we have $\cN_{(1, 0)}(S)=\{ Y=\cY_S\}$.  The RZ-space we considered is $\mathcal N^{\Kra} = \cN_{(1, 0)} \times \cN_{(1, 1)}^{\Kra}$.   Adding $\cN_{(1, 0)}$ enables us to define cycles naturally. For simplicity, we often identify $\mathcal N^{\Kra}$ with $\cN_{(1, 1)}^{\Kra}$ in this paper (ignore the one element formal scheme $\cN_{(1, 0)})$.

For a $\bx \in \bV\setminus\{0\}$, let $\cZ^{\Kra}(\bx)$  be
 the closed (formal) subscheme of $\cN^{\mathrm{Kra}}$ consisting of $(Y, X) =(Y, \iota_Y, \lambda_Y, \rho_Y,  X, \iota_X, \lambda_X, \rho_X,  \cF_X) \in \cN^{\Kra}$ such that $\rho_X^{-1}\circ \bx \circ \rho_Y$ lifts to an $\Oo_F$ homomorphism $x:Y\rightarrow X$.  According to \cite[Proposition 4.3]{Ho2}\footnote{\cite{Ho2} deals with the case when $F_0=\Q_p$. Using the relative display theory and its Grothendieck-Messing theory developed in \cite{ACZ}, the results and proofs there extend word by word to general $F_0$.}, $\cZ^\Kra(\mathbf{x})$ is a Cartier divisor.  The closed formal subscheme $\cZ(\bx) \subset \mathcal N_{(1, 0)} \times \cN_{(1, 1)} $ can be defined the same way. Notice, however, that $\cZ(\bx)$ is not a Cartier divisor.
	For a pair of independent vectors $\bx_1, \bx_2\in \bL$, define the intersection number
\begin{equation}
\hbox{Int}(L_{\bx_1, \bx_2})=\chi(\cN^{\mathrm{Kra}},\Oo_{\cZ^{\mathrm{Kra}}(\bx_1)}\otimes^{\mathbb{L}}\Oo_{\cZ^{\mathrm{Kra}}(\bx_2)}),
\end{equation}
which depends only on the lattice $L_{\bx_1, \bx_2} \subset \bV$ generated by $\bx_1$ and $\bx_2$.
Two Hermitian matrices $T_1, T_2 \in \hbox{Herm}_2(F)$ are said to be equivalent, denoted by $T_1\approx T_2$, if there is a non-singular matrix $g \in \hbox{GL}_2(\Oo_F)$ such that $g^{t} T_1 \bar{g} =T_2$. It is known (see \cite[Proposition 4.3]{J}) that every  Hermitian matrix of order $2$  is  equivalent to either a diagonal matrix or an anti-diagonal matrix of the form $\kzxz {0} {\pi^n} {(-\pi)^n} {0}$  with $n$ odd.
The first  main result of this paper,  together with  that in \cite{Shi2},  is the following theorem.

\begin{theorem}\label{maintheo1} For a pair of independent vectors $\bx_1, \bx_2 \in \bV$, let $T(\bx_1, \bx_2) = ( h(\bx_i, \bx_j))$ be the associated Gram matrix.   Then $\hbox{Int} (L_{\bx_1, \bx_2})$ depends only on the equivalence class of $T=T(\bx_1, \bx_2)$.  More precisely, $\hbox{Int} (L_{\bx_1, \bx_2})=0$ unless $T(\bx_1, \bx_2)$ is integral. Assume that $T(\bx_1, \bx_2)$ is integral.
\begin{enumerate}
\item  When  $T \approx  \hbox{diag}(u_1 (-\pi_0)^\alpha, u_2 (-\pi_0)^\beta) $ with $u_i \in \Oo_{F_0}^\times$ and $0\le \alpha \le \beta$, we have
$$
\mathrm{Int}(L_{\bx_1, \bx_2})=\begin{cases}
    \alpha+\beta-\frac{2\q(\q^\alpha-1)}{\q-1}     &\ff \bV \hbox{ is isotropic},
    \\
     2 \sum_{s=0}^\alpha \q^s(\alpha +\beta +1-2s) -\alpha -\beta -2 &\ff \bV \hbox{ is anisotropic}.
     \end{cases}
$$

\item  When  $ T \approx \kzxz {0} {\pi^n} {(-\pi)^n} {0}$ with $n$ odd (which occurs only when  $\bV$ is isotropic),   we have
$$
\Int(L_{\bx_1, \bx_2}) =
- \frac{(\q+1)(\q^{\frac{n+1}2}-1)}{\q-1} + n +1.
$$
\end{enumerate}
\end{theorem}

The case that $\bV$ is anisotropic was proved by  one of the authors (Y.S.) in \cite{Shi2}. This paper was inspired by his work. A key ingredient is to understand the arithmetic divisor $\cZ^{\Kra}(\bx)$ and in particular the contribution of the exceptional divisors, which is new in the ramified case. In our special  $U(1, 1)$-case, we use critically the special facts:
\begin{enumerate}
\item  $\mathcal N_{(1, 1)}$ descends to a {\bf regular} formal scheme $\cN$ over $\Oo_{\breve{F}_0}$;

\item  The formal scheme $\cN$ is isomorphic to the RZ-space $\mathcal M_{\Gamma_0(\pi_0)}$ (see \cite[Definition 8.1]{RSZ}) in the case $\bV$ is anisotropic (i.e., the quaternion algebra $B$ at the beginning of the introduction is a matrix algebra),  and is isomorphic to the  formal scheme $\mathcal M$, which is represented by the formal completion $\breve{\Omega}$ of  the Drinfeld half plane over $\Oo_{\breve F_0}$ in the case $\bV$ is isotropic (i.e. $B$ is a division algebra)(\cite{D}), see Section \ref{sect:moduli}.
\end{enumerate}
In this paper, we consider the case that $\bV$ is isotropic. Using properties of $\breve\Omega$, we can understand the special fibers of  the divisors $\cZ^{\Kra}(\bx)$. Using explicit local equations at `superspecial points' of $\breve\Omega$, we obtain the following decomposition result of independent interest.
An $\Oo_F$-lattice $\Lambda \subset \bV$  is called a vertex lattice of type $t=0$ or $2$ if $\pi \Lambda \subset  \Lambda^\sharp \subset \Lambda$ with $[\Lambda:\Lambda^\sharp]=t$. Here $\Lambda^\sharp$ is the $\Oo_F$-dual of $\Lambda$ with respect to the Hermitian form on $\bV$. Associated to each vertex lattice $\Lambda_0$ of type $0$  is an exceptional divisor $\Exc_{\Lambda_0} \subset \cN^{\Kra}/\bar\kay$, and associated to each vertex lattice $\Lambda_2$  of type $2$ is a vertical projective line  $\mathbb P_{\bar{\Lambda}_2} /\bar\kay \subset \cN^{\Kra}/\bar\kay$. Here and throughout this paper, we write $\bar{\Lambda} = \Lambda \otimes_{\Oo_F} \bar\kay$ for an $\Oo_F$-lattice $\Lambda$.

\begin{theorem} \label{maintheo2} Let $\bx \in \bV\setminus\{0\}$ with $h(\bx, \bx) \in \Oo_{F_0}$, and let $\cT(\bx)$ be the set of vertex lattices $\Lambda$   such that $\cZ(\bx)^{\Kra}(\bar\kay)$ intersects with $ \mathbb P_{\bar{\Lambda}}$ or $\hbox{Exc}_{\Lambda}$, depending on whether $\Lambda$ is of type $2$ or $0$. Then
$$
\cZ^{\Kra}(\bx) = \sum_{ \Lambda_{2} \in \cT(\bx)} n(\bx, \Lambda_2) \mathbb P_{\bar{\Lambda}_2}
   + \sum_{ \Lambda_{0} \in \cT(\bx)} (n(\bx, \Lambda_0) +1) \Exc_{\Lambda_0} + \cZ^h(\bx).
$$
Here $\cZ^h(\bx)$ is the horizontal component of $\cZ^{\Kra}(\bx)$, which is empty or  non-empty irreducible depending on $h(\bx, \bx) =0$ or not.  $n(\bx, \Lambda)$ is the largest integer $n$ such that $\pi^{-n} \bx \in \Lambda$. When  $h(\bx, \bx) \notin \Oo_{F_0}$, $\mathcal Z^{Kra}(\bx) =0$.
\end{theorem}

The proof of this theorem  occupies Section \ref{sect:decomposition} with   preparation in Section \ref{sect:specialfiber}. Now proof of Theorem \ref{maintheo1} becomes explicit computation of the intersection number among the divisors in  the decomposition above.  This occupies Section \ref{sect:intersection}

To state the last main result of this paper---the Kudla-Rapoport formula, we need to define the local density function. Let $(L, (\, , \,))$ be an integral $\Oo_F$-Hermitian lattice of rank $m$ with a  Gram matrix  $S$, and let $T \in \Herm_n(\Oo_F)$. Let $\mathcal H =\Oo_F^2$ be the Hermitian hyperbolic plane with the standard Hermitian form $(x, y) = \frac{1}{\pi} (x_1 \bar y_2 - x_2 \bar y_1)$ with Gram matrix $\frac{1}{\pi} \kzxz {0} {1} {-1} {0}$. Let $L_r = L \oplus \mathcal H^r$ with Gram matrix $S_r$. Then it can be shown  that there is a unique `local density polynomial' $\alpha(L, T, X)$ of $X=\q^{-2r}$ such that
\begin{equation}
\alpha(L, T, \q^{-2r}) =\int_{\Herm_n(F)} \int_{L_r^n} \psi( \tr (b (T(x) -T))) dx \, db,
\end{equation}
where $T(x) = ( (x_i, x_j))$, $dx$ is the Haar measure on $L_r^n$ with total volume $1$, $db$ is the Haar measures on $\Herm_n(\Oo_F)$ with total volume $1$, and $\psi$ is an additive character of $F_0$ with conductor $\Oo_{F_0}$.  We will denote $\alpha(L, T)=\alpha(L, T, 1)$ so that $\alpha(L, T, \q^{-2r}) = \alpha(L_r, T)$. We will also denote
\begin{equation}
\alpha'(L, T) =-\frac{\partial}{\partial X} \alpha(L, T, X)|_{X=1}.
\end{equation}
Let $\partial_F$ be the different ideal of $F/F_0$. Define
\[\Herm_n(\Oo_F)^\vee:=\{T=(t_{ij})\in \Herm_n(F)\mid \ord_{\pi}(t_{ii}) \ge 0, \hbox{ and } \ord_{\pi}(t_{ij}\partial_F) \ge 0\}.\]
Notice that $\Herm_n(\Oo_F)^\vee$ is the dual of $\Herm_n(\Oo_F)$ under the pairing $\psi(\mathrm{tr}(XY))$.
Simple calculation gives
\begin{equation}
\alpha(L, T) = \q^{-a n (2m -n)} |\{ X \in M_{m,n}(\Oo_F/(\pi_0^a))\mid  X^t S \bar X - T  \in \pi_0^a\cdot \Herm_n(\Oo_F)^\vee\} |
\end{equation}
for sufficiently large integer $a >0$.  For this reason, we will also denote $\alpha(L, T, X)$ by $\alpha(S, T, X)$ and so on.
Now we can state the Kudla-Rapoport formula at the ramified prime as follows, which comes from the comparison between the formula in Theorem  \ref{maintheo1} and the local density formulas in \cite{Shi2} (see Section \ref{sect:localdensity}).

\begin{theorem} \label{maintheo3} Let $L$ be the unimodular $\Oo_F$-Hermitian lattice of rank $2$ as in the beginning of the introduction with Gram matrix $S$, and let $T =T(\bx_1, \bx_2) \in \Herm_2(F)$ with $\bx_i \in  \bV$ being independent. Then $\alpha(L, T) =0$, and
$$
\mathrm{Int}(L_{\bx_1, \bx_2})=2\frac{\alpha'(L,T)}{\alpha(L,S)}-\frac{2\q^2}{\q^2-1}\frac{\alpha(\mathcal H,T)}{\alpha(L,S)}.
$$
\end{theorem}
\begin{remark}
The case that $\bV$ is anisotropic was proved in \cite{Shi2}. In his case,  $\alpha(\mathcal H,T)=0$.
\end{remark}

The paper is organized as follows. In Section \ref{sect:moduli}, we collect the definitions of various relevant moduli spaces and discuss the relations among these spaces. In Section \ref{sect:specialfiber}, we give a description of the special fiber of the moduli spaces and special cycles. In Section \ref{sect:decomposition}, we give a decomposition of the special divisor by studying the local equations of special divisors at superspecial points, and prove Theorem \ref{maintheo2}. Section \ref{sect:intersection} contains a computation of the intersection number between special divisors and the proof of Theorem \ref{maintheo1}. Finally in Section \ref{sect:localdensity}, we review Shi's  local density formula  and prove Theorem \ref{maintheo3}.

For the rest of this paper, we assume that {\bf  $V$ is anisotropic and $\bV$ is isotropic},  i.e., the associated quaternions $B$ being the division algebra, and $\bB =M_2(F_0)$.  Recall that $\Oo_B= \Oo_F + \Oo_F \delta$ with $\delta x =\bar x \delta$ and $\Delta = \delta^2 \in \Oo_{F_0}^\times$. So $E=F_0(\delta)$ is the unique unramified quadratic field extension of $F_0$ with ring of integers $\Oo_E=\Oo_{F_0}(\delta)$, which will be viewed as a subfield of $\breve{F}_0$. Let $\sigma \in  \hbox{Gal}(\breve{F}_0/F_0)$ be the Frobennius element lifting $\sigma(x) = x^q $ of $\mathrm{Gal}(\bar \kay/\kay)$. Denote $W=\Oo_{\breve{F}_0}$. We will also identify
\begin{equation}
B= \{ \kzxz {a} {b \pi_0} {b^\sigma} {a^\sigma}:\, a, b \in E\},
\end{equation}
together with  two embeddings
\begin{equation}
\iota:  E \hookrightarrow B, a \mapsto \kzxz {a} {0} {0} {a^\sigma}, \hbox{ and }  \iota:  F \hookrightarrow B, a +b \pi \mapsto \kzxz {a} {b \pi_0} {b} {a}.
\end{equation}
Note that $\pi a = a^\sigma \pi$ for all $a \in E$.

If $R$ is an $\Oo_E$-algebra,
there is a ring isomorphism
\begin{equation}\label{OoEgrading}
\Oo_E\otimes_{\Oo_{F_0}} R\cong R\times R, \quad
a\otimes x \mapsto (a x,\sigma(a)x).
\end{equation}
For any $\Oo_E\otimes_{\Oo_{F_0}} R$ module $M$, the above ring isomorphism induces a $\Z/2\Z$-grading
\begin{equation}\label{OoEgrading'}
M=M_0\oplus M_1.
\end{equation}

{\bf Acknowledgment:}
We thank Chao Li for his help during preparation of this work. We thank the referee for their helpful comments and suggestion.

\section{Preliminary } \label{sect:moduli}
Following \cite{BC} and \cite{KR3}, we briefly review several  moduli functors in this section for later use.

\subsection{Framing objects and their Dieudonn\'e modules}\label{section of N_F}

Throughout the paper, by relative Dieudonn\'e module we mean relative $\Oo_{F_0}$-Dieudonn\'e module in the sense of \cite[3.56]{RZ} or \cite[Appendix B.8]{fargues2006isomorphisme}, which is a special case of the $\Oo_{F_0}$-display studied in \cite{ACZ}.
Let $\bY$ be the unique 1 dimensional formal $\Oo_{F_0}$-module of relative height 2 over $\bar{\kay}$ as in the introduction. Then its relative Dieudonn\'e module (see page 4 of \cite{KR3} or \cite[Proposition 3.56]{RZ}) $M(\bY)$  is a free $\Oo_{\breve{F}_0}$-module of rank 2. We can choose a basis $\{e,f\}$ of $M(\bY)$ such that
\begin{equation}
Fe=f,\ Ff=\pi_0 e,\ Ve=f,\ V f=\pi_0 e,\ \langle e,f\rangle=\delta. \label{polarization on bY}
\end{equation}
With respect to this basis, we have
\begin{equation}\label{action of F,V on Y}\End (M(\bY))\cong \left\{\left(\begin{array}{cc} a & b \pi_0 \\
b^\sigma & a^\sigma
\end{array}\right)\mid a,b \in \Oo_E \right\}=\Oo_B.\end{equation}
In particular
\begin{equation}\label{action of pi on Y}\iota_{\bY}(\pi)=\begin{pmatrix}0&\pi_0\\1&0\end{pmatrix}.\end{equation}

The framing object $(\bX, \iota_{\bX}, \lambda_\bX)$ in the introduction can be explicitly realized as follows.
Let $\bX=\bY\times \bY$ and identify $\End(\bX)$ with $M_2(\Oo_B)$. Then
we can define an action of $\Oo_B$ on $\bX$ via
\begin{equation}
\iota_\bX(b)=\left(\begin{array}{cc}
b & 0 \\
0 & \pi b \pi^{-1}
\end{array}\right).
\end{equation}
 Finally, we identify $\bX^\vee =\bX$ and define the principal polarization
\begin{equation}\label{polarization of mathbbX}
\lambda_\bX= \kzxz {0} {1} {1} {0}.
\end{equation}
Then  the Rosati involution associated to $\lambda_\bX$ induces the involution $b \mapsto b^*= \pi b^\iota \pi^{-1}$ on $B$. This  triple $(\bX, \iota_\bX, \lambda_\bX)$ is the basic Drinfeld triple defined in \cite[Proposition 1.1]{KR3}.
We obtain the basic framing object $(\bX, \iota_\bX, \lambda_\bX)$ from the above framing object by restricting $\iota_\bX$ on $\Oo_F$. One can check that the Hermitian form on $\bL= \Hom_{\Oo_F} (\bY, \bX)$  induced by $\lambda_{\bX}$  is {\bf isotropic} as required in this paper (see Lemma \ref{lem2.1}).

We choose a basis $\{e_0,f_0,e_1,f_1\}$  of $M(\bX)$, the relative Dieudonn\'e module of $\bX$, such that $\{e_i,f_i\}$ ($i=0,1$) is the basis of $M(\bY)$ for the $i$-th copy of $\bY$. Then we have
\begin{equation}
Fe_i=f_i, \ F f_i=\pi_0 e_i, \ Ve_i=f_i, \ V f_i=\pi_0 e_i.
\end{equation}
\begin{equation}
\iota_{\mathbb{X}}(\pi) e_i=f_i,\  \iota_{\mathbb{X}}(\pi) f_i =\pi_0 e_i, \ \iota_{\mathbb{X}}(\delta) e_i= (-1)^i \delta e_i,\  \iota_{\mathbb{X}}(\delta) f_i= (-1)^{i+1} \delta f_i.
\end{equation}
\begin{remark}
In some literature (e.g. \cite{BC} and \cite{RTW}), the operator $\iota_{\mathbb{X}}(\pi)$ is denoted as $\Pi$. In this paper, since the action of $\Oo_F$ on the relative Dieudonn\'e module  (later Cartier module) is unambiguous, we will mostly write $\iota_{\bX}(a)$  (later $\iota(a)$) simply as $a$ for $a\in \Oo_F$.
\end{remark}
We also need to consider the $\Oo_E$-action on $M(\bX)$ obtained by restricting $\iota_{\bX}$ on $\Oo_E$.
The grading on $M(\bX)$ as defined in \eqref{OoEgrading'} is
\begin{equation}
M_0=M(\bX)_0=\mathrm{span}_{\Oo_{\breve{F}_0}} \{e_0, f_1\}, \ M_1=M(\bX)_1=\mathrm{span}_{\Oo_{\breve{F}_0}} \{e_1, f_0\}.
\end{equation}

Let $N=N(\bX)\coloneqq M(\bX)\otimes \Q$. The principal polarization $\lambda_{\bX}$ induces an alternating form $\langle,\rangle$ on $N$ such that
\begin{align}
&~~~~\langle Fx,y \rangle= \langle x, Vy\rangle^{\sigma},\\
&\langle a x,y\rangle= \langle x, \bar{a} y\rangle, a \in \Oo_F.  \notag
\end{align}
In terms of the $\breve{F}_0$-basis $\{e_0,f_0,e_1,f_1\}$, we have
\begin{equation}
\langle e_0, f_1\rangle=\langle e_1, f_0\rangle=\delta,
\end{equation}
and all the other pairings between the basis vectors are $0$.
Following \cite{KR3}, we can define a Hermitian form $(,)$ on $N$ by
\begin{equation}
(x,y)=\delta[\langle \pi x, y\rangle+\pi \langle  x, y\rangle].   \label{definition of h_F}
\end{equation}
Let $\tau=\pi V^{-1}$. Then  $C\coloneqq N^\tau=\mathrm{span}_{F}\{e_0,e_1\}$ is  a Hermitian $F$-vector space with $C \otimes_F \breve F= N$. Moreover,
\begin{equation}\label{h_Fbasis}
(e_0,e_1)=-\delta^2,\ (e_0,e_0)=(e_1,e_1)=0.
\end{equation}
So   $(C, (,))$ is isotropic, satisfying the assumption of this paper. There is a similarly defined Hermitian form $(,)_{\bY}$ on $N(\bY)^{\tau}$ given by
\begin{equation}
(x,y)_{\bY}=\delta [\langle \pi x,y\rangle +\pi \langle x,y\rangle ].  \label{def of h_bY}
\end{equation}
Here $\langle e,f\rangle =\delta$ as in \eqref{polarization on bY}.

Recall $\mathbb{L}=\Hom_{\Oo_F}(\bY,\bX)\cong \Hom_{\Oo_{F}}(M(\bY),M(\bX))$ and $\bV= \bL \otimes_{\Oo_F} F$. For  $\bx \in \bL$,  we will also  use $\bx$ to denote the corresponding homomorphism between relative Dieudonn\'e modules when there is no confusion. Now $\bx(e)\in M(\bX)^{\tau}$ since $\tau(e)=e$ by \eqref{action of F,V on Y}, \eqref{action of pi on Y} and $\bx$ commutes with $\pi$ and $ V$. This way, we obtain a map from $\bL$ to $M(\bX)^{\tau}$ by sending $\bx\in\bV$ to $\bx(e)$.
\begin{lemma} (\cite[Lemma 3.3]{Shi1}) \label{lem2.1}
	We can identity $\bL$ with $M(\bX)^{\tau}$ by sending $\bx \in \bL$ to $\bx(e)$. Moreover,
	\begin{equation}\label{comparison between hermitian form}h(\bx,\bx)(e,e)_{\bY}=(\bx(e),\bx(e)),
	\end{equation}
	where $h(,)$ is defined in \eqref{eq:V-Hermitian}.
	According to \eqref{def of h_bY}, $(e,e)_{\bY}=-\delta^2$. In particular, $h(,)$ is also isotropic since $(,)$ is isotropic as we showed previously.
\end{lemma}

Because of Lemma \ref{lem2.1}, we will often identify $\bV$ with $C$, and $\bL$ with $M(\bX)^\tau$, via $\bx \mapsto \bx(e)$. The Hermitian form is shifted by a factor $-\delta^2$.

\subsection{Moduli space $\cM$}\label{subsection of M}

We first recall the following result of Kramer (\cite{Kr}).
\begin{proposition}
	$\cN^{\Kra}$ is representable by a formal scheme over $\Spf \Oo_{\breve{F}}$ which is regular and has semi-stable reduction. The natural forgetful map
	\begin{align*}
	\Phi:\cN^{\Kra}&\rightarrow \cN_{(1,1)}\\
	(X,\iota,\lambda,\rho,\mathcal{F}_X)&\mapsto (X,\iota,\lambda,\rho)
	\end{align*}
	is an isomorphism outside the singular locus $\mathrm{Sing}$ of $\cN_{(1,1)}$, and $\Phi^{-1}(x)$ for $x\in \mathrm{Sing}$ is an exceptional divisor which is isomoprhic to the projective line  $\bP_{\bar\kay}^{1}$.
\end{proposition}
By \cite[Proposition 2.7]{Shi2}, $\cN^\Kra$ is the blow up of $\cN_{(1,1)}$ along $\mathrm{Sing}$.

As mentioned in the introduction, the formal scheme $\cN_{(1, 1)}$ descends to a formal scheme $\cN$ over $\Oo_{\breve F_0}$ (since the Kottwitz determinant condition is defined over $\Oo_{\breve F_0}$). A beautiful result of Kudla and Rapoport (\cite{KR3}) asserts that $\cN$ is actually isomorphic to the modular functor $\cM$ which is represented by  the base change  $\breve{\Omega}$ of  the formal completion of the Drinfeld's upper half plane. The observation is a key ingredient in obtaining the main results of this paper. We now briefly review $\cM$ and $\breve{\Omega}$ in this subsection and next one.

 A formal $\Oo_B$-module over an $\Oo_{F_0}$-algebra $R$ is a formal $\Oo_{F_0}$-module X over $R$ with an action $\iota_{B}:\Oo_B\rightarrow \End (X)$ extending the action of $\Oo_{F_0}$. X is called {\bf special} if the action of $\Oo_E\subset \Oo_B$ makes $\mathrm{Lie}X$ a free $R\otimes_{\Oo_{F_0}}\Oo_E$-module of rank one. The $R$-module $\mathrm{Lie}(X)$ is $\Z/2\Z$-graded by the action of $\Oo_E $:
\begin{align*}
&(\mathrm{Lie}X)_0=\{x\in\mathrm{Lie}(X)\mid \iota_{\Oo_B}(a)m=am~ \text{for~all~$a\in\Oo_E$}\},\\
&(\mathrm{Lie}X)_1=\{x\in\mathrm{Lie}(X)\mid \iota_{\Oo_B}(a)m=\sigma(a)m~ \text{for~all~$a\in\Oo_E$}\}.
\end{align*}
Then X is special if $\mathrm{Lie}(X)_i$ is a free $R$-module of rank one. Over $\bar{\kay}$, there is a unique special formal $\Oo_B$-module up to $\Oo_B$-linear isogeny, which is  $(\bX,\iota_{\bX})$ in Section \ref{section of N_F}.

\begin{definition}
	We define a moduli functor $\mathcal{M}$ on $\mathrm{Nilp}_{\Oo_{\breve{F}_0}}$ that associates $S$ with the set of isomorphism classes of triples $(X,\iota_{B},\rho)$, where
	\begin{itemize}
		\item $(X,\iota_B)$ is a special formal $\Oo_B$-module of dimension 2 and relative height 4 over $S$.
		\item  $\rho$ is a $\Oo_B$-linear quasi-isogeny
		\[\rho: X\times_S\bar{S}\rightarrow\bX\times_{\Spec \bar{\kay}}\bar{S}   \]
		of height 0. Here $\bar S$ is the special fiber of $S$.
	\end{itemize}
\end{definition}

The following is a result of Drinfeld that shows the automatic existence of polarizations on special formal $\Oo_B$-modules (see Proposition 1.1 of \cite{KR3}).
\begin{proposition}[Drinfeld]\label{drinfeld's proposition}
	Let $\pi\in \Oo_{B}$ be a uniformizer such that $\pi^2=\pi_0$, and consider the involution $b\rightarrow b^{*}=\pi b' \pi^{-1}$ of $B$, where $b\rightarrow b'$ denotes the main involution.
	\begin{enumerate}[label=(\roman *)]
		\item On $\bX$ there exists a principal polarization $\lambda^{0}_{\bX}:\bX \xrightarrow{\sim} \bX^{\vee}$ with associated Rosati involution $b\rightarrow b^{*}$. Furthermore, $\lambda^{0}_{\bX}$ is unique up to a factor in $\Oo_{F_0}^{\times}$.
		\item Fix $\lambda^0_{\bX}$ as in $(i)$. Let $(X,\iota,\rho)\in \mathcal{M}(S)$, where $S\in \mathrm{Nilp}_{\Oo_{\breve{F}_0}}$. On $X$ there exists a unique principal polarization $\lambda ^{0}_{X}:X\xrightarrow{\sim}X^{\vee}$ making the following diagram commutative,
		\[ \begin{tikzcd}
		X\times_{S}\bar{S} \arrow{r}{\lambda^0_X} \arrow{d}{\rho} & X^{\vee}\times_{S}\bar{S} \\
		\bX\times_{\Spec \bar\kay}\bar{S} \arrow{r}{\lambda^0_{\bX}}& \bX^{\vee}\times_{\Spec \bar\kay}\bar{S}\arrow{u}{\rho^{\vee}}
		\end{tikzcd}
		\]
		
	\end{enumerate}
\end{proposition}

\begin{theorem}(\cite[Theorem 1.2]{KR3}) \label{alternative description F}
	Assume that $p\neq 2$. The morphism of functors on $\mathrm{Nilp}_{\Oo_{\breve{F}_0}}$ given by $(X,\iota_B,\rho)\rightarrow (X,\iota,\lambda_X^{0},\rho)$ induces an isomorphism of formal schemes
	\[\eta_F : \mathcal{M}\xrightarrow{\sim}\cN.\]
	Here  $\iota$ is the restriction of $\iota_{B}$ to $\Oo_F$ and $\lambda_X^0$ is the principal polarization given by Proposition \ref{drinfeld's proposition}.
\end{theorem}

\subsection{ Drinfeld upper half plane and its formal completion}

For convenience of the reader and to set up notation for the rest of the paper, we briefly review the well-known facts about the Drinfeld upper half place ${\Omega}$,  its formal completion $\widehat{\Omega}$, and its base change  $\breve{\Omega}$ to $W=\Oo_{\breve F_0}$, following \cite{BC}.

Recall that the Bruhat-Tits tree $\mathcal B(\hbox{PGL}_2(F_0))$ consists of vertices and edges. The vertices are given by the homothety classes $[\Lambda]$ of $\Oo_{F_0}$-lattices in $F_0^2$, and the edges are given by pairs $([\Lambda], [\Lambda'])$ of the homothety classes such that $\pi_0 \Lambda' \subset \Lambda \subset \Lambda'$ for suitable choices of lattices $\Lambda$ and $\Lambda'$ in their homothety classes. We then say $\Lambda$ and $\Lambda'$ are adjacent. We use $\mathbb P_{[\Lambda]}$ to denote the projective line over $\Oo_{F_0}$ associated to $\Lambda$ depending on its homothety class.
Let
$$
\Omega_{[\Lambda]} = \mathbb P_{[\Lambda]} - \mathbb P_{[\Lambda]}(\kay)
$$
be the  projective line with the $q+1$ rational $\kay$-points removed. When  $([\Lambda], [\Lambda'])$ is an edge, $\Lambda \pmod{\pi_0}$ gives a $\kay$-rational point in $\mathbb P_{[\Lambda']}$. We write $\bP _{[\Lambda,\Lambda']}$ for the blow up of $\bP_{[\Lambda']}$ at this point, which is isomorphic to the blow up of $\bP_{[\Lambda]}$ at the rational point determined by $\Lambda'$. We write $\Omega_{[\Lambda,\Lambda']}$ for the complement of the nonsingular rational points of the special fiber of $\bP _{[\Lambda,\Lambda']}$. There is an open embedding $\Omega_{[\Lambda]} \hookrightarrow \Omega_{[\Lambda,\Lambda']}$.

Define $\widehat{\Omega}_{[\Lambda]}$ and $\widehat{\Omega}_{[\Lambda,\Lambda']}$ to be the formal completion of $\Omega_{[\Lambda]}$ and $ \Omega_{[\Lambda,\Lambda']}$ along their special fibers respectively. For two different edges with a common vertex $[\Lambda]$, we can glue  $\widehat{\Omega}_{[\Lambda, \Lambda']}$ and  $\widehat\Omega_{[\Lambda, \Lambda'']}$  along $\widehat\Omega_{[\Lambda]}$. This gives the formal scheme $\widehat{\Omega}$ by glueing over $\mathcal B(\hbox{PGL}_2(F_0))$. The generic fiber of $\widehat{\Omega}$ is the Drinfeld $p$-adic half space $\Omega=\bP^1(\mathbb{C}_p)- \bP^1 (F_0)$, where $\mathbb{C}_p$ is the completion of $\bar{F}_0$.
Define $\breve{\Omega}=\widehat{\Omega}\times_{\Spf \Oo_{F_0}}\Spf \Oo_{\breve{F}_0}$. Similarly define $\breve{\Omega}_{[\Lambda]}$ and $\breve{\Omega}_{[\Lambda,\Lambda']}$. Drinfeld's well-known result
(see  \cite{D} and  \cite{BC})  asserts that $ \mathcal{M}$ is represented by $\breve{\Omega}$. The following proposition will be used in Section 4.

\begin{proposition}(Deligne functor) (\cite[Propositions 4.2, 4.4]{BC}) \label{Deligne functor rep at superspecial point}
\begin{enumerate}
\item
$\widehat{\Omega}_\Lambda$ represents the functor that associates an $\Oo_{F_0}$-algebra $R\in \mathrm{Nilp}_{\Oo_{F_0}}$ the collection of isomorphism classes of pairs $(\mathcal{L},\alpha)$, where $\mathcal{L}$ is a free $R$-module of rank 1 and $\alpha:\Lambda \rightarrow \mathcal{L}$ is a homomorphism of $\Oo_{F_0}$-modules satisfying the condition:
	\begin{itemize}
		\item for all $x\in \Spec(R/\pi_0R)$, the map $\alpha(x):\Lambda /\pi_0\Lambda  \rightarrow \mathcal{L}\otimes_R k(x)$ is injective.
	\end{itemize}

 \item
	 $\widehat{\Omega}_{[\Lambda ,\Lambda ']}$ represents the functor that associates an $\Oo_{F_0}$-algebra  $R\in \mathrm{Nilp}_{\Oo_{F_0}}$ the collection of isomorphism classes of commutative diagrams:
	 \[ \begin{tikzcd}
	 \pi_0 \Lambda'  \arrow[hookrightarrow]{r}{} \arrow{d}{\alpha'/\pi_0} & \Lambda  \arrow{d}{\alpha} \arrow[hookrightarrow]{r}{}& \Lambda' \arrow{d}{\alpha'}  \\
	 \mathcal{L}' \arrow{r}{c'}& \mathcal{L}\arrow{r}{c}&\mathcal{L}'
	 \end{tikzcd}
	 \]
	 where $\mathcal{L}$ and $\mathcal{L}'$ are free R-modules of rank 1, $\alpha$ and $\alpha'$ are the homomorphisms of $\mathcal{O}_{F_0}$-modules, $c$ and $c'$ are homomorphisms of $R$-modules, satisfying the conditions:
	 for all $x\in \Spec (R/ \pi_0 R)$, one has
	 \begin{itemize}
	 	\item ker$(\alpha'(x)$: $\Lambda' /\pi_0 \Lambda ' \rightarrow \mathcal{L}'\otimes_{R} k(x))\subset \Lambda /\pi_0 \Lambda' ,$
	 	\item ker$(\alpha(x)$: $\Lambda /\pi_0 \Lambda  \rightarrow \mathcal{L}\otimes_{R} k(x))\subset \pi_0\Lambda' /\pi_0 \Lambda. $
	 \end{itemize}
\end{enumerate}

\end{proposition}

We now describe the explicit equations of $\widehat{\Omega}_{[\Lambda]}$ and $\widehat{\Omega}_{[\Lambda,\Lambda']}$.
Let $\{e_1,e_2\}$ be a basis of $\Lambda$. Then we have
\begin{equation} \label{eq:vertex}
\widehat{\Omega}_{[\Lambda]}=(\bP_{[\Lambda]}- \bP_{[\Lambda]}(k))^{\wedge}=\Spf \Oo_{F_0}[T,(T^{\q}-T)^{-1}]^{\wedge},
\end{equation}
where $T=\frac{X_0}{X_1}$ and $X_i$ is the coordinate of $\bP_{[\Lambda]}$ with respect to the basis $\{e_1,e_2\}$ and $``$hat$"$ indicates completion along the special fiber.

Without loss of generality, we can assume $\Lambda '$ has a basis $\{e_1,\pi_0^{-1}e_2\}$. Then we have
\begin{equation} \label{eq:edge}
\widehat{\Omega}_{[\Lambda,\Lambda']}=\Spf \Oo_{F_0}[T_0,T_1,(T_0^{\q-1}-1)^{-1},(T_1^{\q-1}-1)^{-1}]^{\wedge}/(T_0T_1-\pi_0).
\end{equation}
In this case, the open immersions $\widehat{\Omega}_{[\Lambda]}\hookrightarrow \widehat{\Omega}_{[\Lambda,\Lambda']} $ and $\widehat{\Omega}_{[\Lambda']}\hookrightarrow \widehat{\Omega}_{[\Lambda,\Lambda']} $ are induced by
\begin{equation}  \label{eq:embedding}
T_0\mapsto  T,\ T_1\mapsto \pi_{0} \cdot T^{-1},
\end{equation}
and
\begin{equation}  \label{eq:embeddingII}
	T_0\mapsto  \pi_{0} \cdot T^{-1},\ T_1\mapsto  T.
\end{equation}

\section{Special fiber of special cycle} \label{sect:specialfiber}
In this section, we study the support of the special cycles. Essentially we only need the Pappas model $\cN_{(1,1)}$ (see Proposition \ref{prop Phi}).
\subsection{Special fiber of $\cN_{(1,1)}$.} Since $\Oo_F/(\pi) =\Oo_{F_0}/(\pi_0) =\kay$, we have $\cN_{(1, 1)}/\bar\kay \cong \cN/\bar\kay$. We briefly review the structure of $\cN_{(1, 1)}/\bar{\kay}$ following \cite{KR3}.

Recall that $C=N^{\tau}, \tau=\pi V^{-1}$. For an $\Oo_F$-lattice $\Lambda_F$ in $C$, set
\[\Lambda_{F}^{\sharp}=\{x\in C\,\mid (x,\Lambda_F)\, \subset \Oo_F\}.\]
Similarly, for a lattice $\Oo_{\breve{F}}$-lattice $M\subset N$, set
\[M^{\sharp}=\{x\in N\,\mid (x,M)\, \subset \Oo_{\breve{F}}\}.\]
\begin{definition}\label{def of vertex O_F lattice}
	For an $\Oo_F$-lattice $\Lambda$ in $C$, $\Lambda$ is called a vertex lattice of type $t$ if $\pi\Lambda \subset \Lambda^{\sharp} \stackrel{t}{\subset} \Lambda$, meaning $[\Lambda : \Lambda^{\sharp}]=t$. In our case $t=0$ or $2$ by Lemma 3.2 of \cite{RTW}.
\end{definition}
\begin{remark}
In the rest of the paper we will often use the subscript $0$ or $2$ to indicate the type of a vertex $\Oo_F$-lattice. For example $\Lambda_0$ often stands for a vertex lattice of type $0$.
\end{remark}

Let $\mathcal B= \mathcal{B} (\mathrm{PU}(C))$ be the Bruhat-Tits tree of $\hbox{PU}(C)$---the projective unitary group of $C$. Its vertices correspond to vertex $\Oo_F$-lattices $\Lambda_2$ of type 2, and  its edges correspond to vertex $\Oo_F$-lattices $\Lambda_0$ of type $0$.  Two vertices  $\Lambda_2$ and $\Lambda_2'$  are connected by an edge $\Lambda_0$, or adjacent, if and only if $\Lambda_2\cap\Lambda_2'=\Lambda_0$. The  following lemma is easy to check (recall that $C$ is isotropic) and is left to the reader.

\begin{lemma} \label{lem3.2}
Let $\Lambda_0$ be a vertex $\Oo_F$-lattice of type $0$. There is an $\Oo_F$-basis $\{w_0, w_1\}$ of $\Lambda_0$ with Gram matrix $\kzxz {0} {1} {1} {0}$. There are exactly two vertex $\Oo_F$-lattices of type $2$ containing $\Lambda_0$:
$$
\Lambda_2 =\Oo_F \pi^{-1} w_0 + \Oo_F w_1, \quad \hbox{and}\quad \Lambda_2'= \Oo_F w_0 + \Oo_F \pi^{-1} w_1.
$$
There are $\q+1$ adjacent (type $2$) vertices of $\Lambda_2$ in $\mathcal B$ , and they are
$$
\Lambda_{\infty} =\Oo_F  \pi^{-2} w_0 + \Oo_F \pi w_1,\quad \Lambda_{k}=\Oo_F w_0+\Oo_F \pi^{-1}(k\pi^{-1}w_0+w_1)
$$
 where $k$ runs through the representatives of $\Oo_{F_0}/ (\pi_0)$.
\end{lemma}

Recall the following results.
\begin{proposition} (\cite[Proposition 2.2]{RTW} and \cite[Lemma 3.2]{KR3}) \label{k bar points of N}  Let
$\mathcal{L}(N) $ be the set of $\Oo_{\breve F}$-lattices
\[
\{M\subset N\,\mid  \pi_0 M \subset VM \stackrel{2}{\subset} M, \quad M^{\sharp}=M, \quad  \dim_{\bar\kay} VM/(VM \cap \pi M) \le 1, \}.
\]
Then
the map
$$
\cN_{(1, 1)}(\bar\kay) \rightarrow \mathcal L(N), \quad x=(X, \iota, \lambda, \rho) \mapsto M(x):=\rho (M(X))\subset N
$$
is a bijection. Moreover, for $M =M(x) \in  \mathcal L(N)$, we have
	\begin{enumerate}[label=(\roman*)]
\item If $M$ is $\tau$-stable, then $M$ is of the form $M=\Lambda_0\otimes_{\Oo_{F}}\Oo_{\breve{F}}$ for some vertex $\Oo_F$-lattice $\Lambda_0$ of type $0$ in $C$,
\item If $M$ is not $\tau$-stable, then
		\[M+\tau M=\Lambda_2\otimes_{\Oo_{F}}\Oo_{\breve{F}}\]
	for some vertex $\Oo_F$-lattice $\Lambda_2$  of type 2 in $C$.
	\end{enumerate}
\end{proposition}

\begin{proposition}\label{proposition about reduced locus of N_F}  As in the introduction,  we use $\bP_{\bar{\Lambda}_2}$ to denote the projective line over $\bar{\kay}$ associated to  $\Lambda_2 \otimes_{\Oo_F} \bar\kay$.

(1)  For every vertex  $\Oo_F$-lattice $\Lambda_2$ of type $2$, there is a closed immersion over $\bar\kay$
$$
i_{\bar{\Lambda}_2}:\mathbb P_{\bar{\Lambda}_2} \rightarrow \cN_{(1, 1)}/\bar\kay ,
$$
 which is given on $\bar\kay$-points as follows: it sends a line $l \subset \Lambda_2 \otimes_{\Oo_F} \bar\kay$ to  its preimage   under the projection  $\Lambda_2 \otimes {\mathcal{O}}_{\breve{F}} \rightarrow(\Lambda_2/\pi \Lambda_2) \otimes_{\Oo_F} \bar\kay$. Moreover,  $i_{\Lambda_2}(l)$ is $\tau$-invariant if and only if $l$ is $\kay$-rational, i.e., $l$ is the extension of a line in $\Lambda_2/\pi \Lambda_2$.
	
(2)  Let $\mathbb P_{\bar{\Lambda}_2}$ denote also its image in $\cN_{(1, 1)}/\bar\kay$ under     $i_{\Lambda_2}$. Then
$$
\cN_{(1, 1)}/\bar\kay = \bigcup_{ \Lambda_2 } \mathbb P_{\bar{\Lambda}_2},
$$
where the union is over all vertex lattices of type $2$.  Two such lines $\mathbb P_{\bar{\Lambda}_2} $ and $ \mathbb P_{\bar{\Lambda}_2'}$  intersect if and only if $\Lambda_2$ and $\Lambda_2'$ are connected by an edge, i.e.,  $\Lambda_2 \cap \Lambda_2' =\Lambda_0$ is a vertex $\Oo_F$-lattice of type $0$.  In such a case,   the two projective lines intersect at exactly one $\kay$-rational point, denoted by $pt_{\Lambda_0}$. Such a point is called   a  superspecial point in $\cN_{(1, 1)}$.

(3) The singular locus $\mathrm{Sing}$ of $\mathcal N_{(1, 1)}$ consists of all the superspecial   $pt_{\Lambda_0}$ with  $\Lambda_0$ running through all   vertex $\Oo_F$-lattices of  type  $0$.
\end{proposition}
\begin{proof} For the proof of $(1)$, see \cite[Lemma 3.3, Proposition 3.4]{KR3}. For $(2)$, if a point $x$ lies in the intersection between $\mathbb P_{\bar\Lambda_2} $ and $ \mathbb P_{\bar\Lambda_2'}$, then $M(x)\subset \Lambda_2\otimes_{\Oo_F} \Oo_{{\breve{F}}} \cap  \Lambda_2'\otimes_{\Oo_F} \Oo_{{\breve{F}}}$ by the description of $i_{\bar\Lambda_2}$ in $(1)$. Together with the fact $M(x)=M(x)^{\sharp}$, this implies $\Lambda_2\cap \Lambda_2'=\Lambda_0$, where $\Lambda_0$ is some vertex lattice of type $0$. The converse is straightforward. Part $(3)$ is immediate from part $(2)$ and the description of the singular locus of $\cN_{(1,1)}$ by local model (see for example \cite[Theorem 4.5]{P}).
\end{proof}

Recall that the blow up $\Phi:\cN^{\Kra}\rightarrow \cN_{(1,1)}$ is an isomorphism outside the singular locus Sing.
\begin{proposition}\label{prop Phi}
(1) For a type $2$  lattice $\Lambda_2$, the Zariski closure of $\Phi^{-1}(\bP_{\bar\Lambda_2}\setminus Sing)$ is a projective line over $\bar\kay$ which we still denote by $\bP_{\bar\Lambda_2}$.

(2) For a vertex lattice $\Lambda_0$ of type $0$, $\Phi^{-1}(pt_{\Lambda_0})$ is a projective line over $\bar\kay$ which we denote by $\Exc_{\Lambda_0}$.

(3) On the special fiber of $\cN^{\Kra}$, two different lines $\bP_{\bar{\Lambda}_2}$ and $ \bP_{\bar{\Lambda}_2'}$ never intersect, and two different lines $\Exc_{\Lambda_0}$ and $\Exc_{\Lambda_0'}$ never intersect. $\bP_{\bar{\Lambda}_2} $ and $\Exc_{\Lambda_0}$ intersect at one point if $\Lambda_0\subset \Lambda_2$, otherwise they do not intersect.
\end{proposition}
\begin{proof}
The above facts will be clear after the discussion in Section \ref{coordinatechartsKramermodel}.
\end{proof}

\begin{figure}[h]
    \centering
\begin{tikzpicture}[scale=0.25]
\draw [color=blue] [rotate=45](0,0) circle (4);
\draw [color=red][rotate=45] (7,0) circle (3);
\draw [color=red][rotate=45] (-7,0) circle (3);
\draw [color=red][rotate=45] (0,7) circle (3);
\draw [color=red] [rotate=45](0,-7) circle (3);
\draw [color=blue][rotate=45] (12,0) circle (2);
\draw [color=blue] [rotate=45](-12,0) circle (2);
\draw [color=blue][rotate=45] (0,-12) circle (2);
\draw [color=blue] [rotate=45](0,12) circle (2);
\draw [color=red][rotate=45] (15,0) circle (1);
\draw [color=red][rotate=45] (-15,0) circle (1);
\draw [color=red][rotate=45] (0,15) circle (1);
\draw [color=red][rotate=45] (0,-15) circle (1);
\draw [color=red] [rotate=45](12,3) circle (1);
\draw [color=red][rotate=45] (12,-3) circle (1);
\draw [color=red][rotate=45] (3,12) circle (1);
\draw [color=red] [rotate=45](-3,12) circle (1);
\draw [color=red] [rotate=45](-12,3) circle (1);
\draw [color=red] [rotate=45](-12,-3) circle (1);
\draw [color=red][rotate=45] (3,-12) circle (1);
\draw [color=red][rotate=45] (-3,-12) circle (1);
\draw [dotted] [rotate=45](16,0)--(18,0);
\draw [dotted][rotate=45] (-16,0)--(-18,0);
\draw [dotted][rotate=45] (12,4)--(12,6);
\draw [dotted][rotate=45] (12,-4)--(12,-6);
\draw [dotted][rotate=45] (-12,4)--(-12,6);
\draw [dotted][rotate=45] (-12,-4)--(-12,-6);
\draw [dotted] [rotate=45](0,16)--(0,18);
\draw [dotted][rotate=45] (4,12)--(6,12);
\draw [dotted] [rotate=45](-4,12)--(-6,12);
\draw [dotted] [rotate=45](0,-16)--(0,-18);
\draw [dotted] [rotate=45]
(4,-12)--(6,-12);
\draw [dotted][rotate=45] (-4,-12)--(-6,-12);
\node at (0,0) [color=blue] {$\bP _{\bar{\Lambda}_2}$};
\node at (5.3,5) [color=red] {$\Exc _{\Lambda_0}$};
\node at (8.57,8.57) [color=blue] {$\bP _{\bar{\Lambda}'_2}$};
\end{tikzpicture}
\caption{Special fiber of $\cN^{\Kra}$ when $\q=3$.}
    \label{specialfiberpic}
\end{figure}
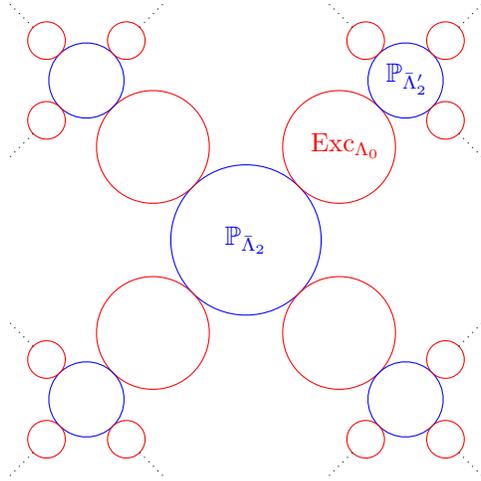

A superspecial point $pt_{\Lambda_0}$ belongs to $\cZ(\bx)$ if and only if $\Exc_{\Lambda_0}\subseteq \cZ^{\Kra}(\bx)$. So in the rest of the section, we only need to study $\cZ(\bx)$.

\subsection{Bruhat-Tits trees of special cycles: rank $1$ case}

For an $\bx \in \bV$, let $\mathcal{T} (\mathbf{x})$ be as in  Theorem \ref{maintheo2}. Alternatively, in terms of $\cZ(\bx)$ (instead of $\cZ^{\Kra}(\bx)$), we have
\begin{equation*}
    \cT(\bx)=\{\Lambda_2 \text{ is of type } 2\mid \bP_{\bar\Lambda_2} \cap \cZ(\bx)\neq \emptyset\}\cup \{\Lambda_0 \text{ is of type } 0\mid pt_{\Lambda_0} \in \cZ(\bx)\}.
\end{equation*}
We will also view a lattice in $\mathcal T(\bf x)$ as a vertex or an edge in $\mathcal B$ depending on whether it is of type $2$ or $0$. We will see  that $\mathcal T(\bx)$ is a tree by Corollaries \ref{ballcorollary} and \ref{cor:cone}.

\begin{lemma} \label{lem3.6} Let $\bx \in \bV$. Then
	$$\mathcal{Z}(\mathbf{x})(\bar{\kay})\cap \mathbb{P}_{\bar{\Lambda}_2}(\bar{\kay})=\left \{\begin{array}{cc}
	\mathbb{P}_{\bar{\Lambda}_2}(\bar{\kay}),     &\text{if}\ \mathbf{x}(e)\in \pi\Lambda_2,  \\
	\text{a single point},  &\text{if}\   \mathbf{x}(e)\in \Lambda_2 \setminus \pi\Lambda_2,  \\
	\emptyset    , &\text{if}\  \mathbf{x}(e) \notin \Lambda_2.
	\end{array}
	\right.
	$$
In particular,  $\Lambda_2 \in \mathcal T(\bx)$ if and only if $\bx(e) \in \Lambda_2$.
\end{lemma}
\begin{proof}
Let $x=(X,\iota,\lambda,\rho)\in \cN_{(1,1)}(\bar\kay)$, let $M(x)=\rho(M(X))\subset N$ as in Proposition \ref{k bar points of N}. Then
	\begin{alignat*}{2}
	x\in \mathcal{Z}(\mathbf{x})(\bar{\kay})\iff & \mathbf{x}(M(\mathbb{Y}))\subset M(x)             \\
	\iff & \mathbf{x}(e)\in M(x),\ \mathbf{x}(f)\in M(x)\\
	\iff & \mathbf{x}(e)\in M(x)
	\end{alignat*} since $\mathbf{x}(e)\in M(x)$ implies $\mathbf{x}(f)=\bx (Ve)\in VM(x)\subset M(x)$.
	\newline
	On the other hand,
	\begin{alignat*}{2}
	x\in \mathbb{P}_{\bar{\Lambda}_2}(\bar{\kay})\iff &  \pi\Lambda_2 \otimes_{\mathcal{O}_F} {\mathcal{O}}_{\breve F} \subset M(x)\subset \Lambda_2 \otimes_{\mathcal{O}_{ F}} { \mathcal{O}}_{\breve F}.
	\end{alignat*}
If $\mathbf{x}(e) \in \pi\Lambda_2$, then  $\mathbf{x}(e) \in M(x)$ for any $x$ such that $\pi\Lambda_2 \subset M(x) \subset \Lambda_2 $.  This implies $\bP_{\bar{\Lambda}_2}(\bar \kay)  \subset \cZ(\bx)(\bar\kay)$.

If $\mathbf{x}(e)\in \Lambda_2 \setminus \pi\Lambda_2 $, then the image of $\mathbf{x}(e)$ is contained in exactly one line in $\bar{\Lambda}_{F, 2}$ and thus gives a single point in  $\mathbb P^{1}_{\bar{\Lambda}_2}(\bar\kay)$.

 Finally, if $\mathbf{x}(e)\not \in \Lambda_2$, it can not lie in any sub-lattice of $\Lambda_2\otimes_{\mathcal{O}_F}\mathcal{O}_{
 \breve{F}}$.
	
\end{proof}

From now on, fix $0\ne \bx \in \bV$, and let $b=\mathbf{x}(e), \ q(b)=(b,b)= -\delta^2 h(\bx, \bx)$.

\begin{lemma}\label{unique}
	Assume $q(b) \ne 0$. Then there exist a unique vertex $\Oo_F$-lattice $\Lambda_b$ of type $0$  such that $ \pi^{-\ord_{\pi_0}(q(b))} b\in \Lambda_b \setminus\pi \Lambda_b$.
\end{lemma}
\begin{proof} Replacing $b$ by $\pi^{-\ord_{\pi_0}(q(b))} b$ if necessary, we may assume $\ord_{\pi_0}(q(b)) =0$, i.e.,
 $(b,b)=u_0\in \mathcal{O}_{F_0}^{\times}$. Write $(F b)^\perp =F c$, then $C =F b \oplus F c$. Since $q(b)$ is a unit, by \cite[Proposition 4.2]{J}, every vertex lattice of type $0$ containing $b$ has the orthogonal decomposition
 $$
 \Lambda =\Oo_F b + \Lambda_1, \quad \Lambda_1= \Lambda \cap F c =\Oo_F b'
 $$
  for some $b' =r c$. $\Lambda$ is a vertex lattice of type $0$ if and only if $q(b') = r \bar r q(c)$ is a unit. Such an $r$ exists and is unique up to a unit in $\Oo_F^\times$. So the vertex $\Oo_F$-lattice $\Lambda_b$ of type $0$ such that $b \in \Lambda_b\setminus\pi\Lambda_b$ exists and is unique.

\end{proof}

\begin{definition}\label{def of distance between type 0 and type 2 lattice}Assume that each edge in $\mathcal{B}$ has length $1$. Let $\Lambda$ and $\Lambda'$ be two vertex $\Oo_F$-lattices. The distance $d(\Lambda, \Lambda')$ is defined to be
\begin{enumerate}
\item the distance between the two vertices in $\mathcal B$ if both are of type $2$,

\item the distance between the vertex $\Lambda$ and the midpoint of the edge $\Lambda'$ in $\mathcal B$ if $\Lambda$ is of type $2$ and $\Lambda'$ is of type $0$,

\item  the distance between the midpoints of the edges $\Lambda$ and $\Lambda'$ in $\mathcal B$ if both are of type $0$.
\end{enumerate}
\end{definition}

For a vertex lattice $\Lambda$, we define
\begin{equation} \label{eq:n}
n(b,\Lambda)\coloneqq \mathrm{max}\{n\in \mathbb{Z} \mid \pi^{-n}b\in \Lambda \}.
\end{equation}
It is easy to check that for
 $\Lambda_2\cap \Lambda'_2=\Lambda_0$,
\begin{equation} \label{n(b,Lambda_0)=min}
n(b, \Lambda_0) = \hbox{min} (n(b,\Lambda_2), n(b, \Lambda'_2)).
\end{equation}
We have the following reformulation of Lemma \ref{lem3.6}.
\begin{lemma}\label{lem:meaning of n(b,Lambda)}
If $b=\bx(e)$,
$n(b,\Lambda)\geq 0$ if and only if $\Lambda\in \cT(\bx)$ .
\end{lemma}

\begin{lemma}\label{bruhat2} Assume $q(b) \ne 0$.
\begin{enumerate}
\item If  $\Lambda$ is a vertex $\Oo_F$-lattice of type $2$, then
  $$
  n(b,\Lambda)=\ord_{\pi_0}(q(b))-d(\Lambda,\Lambda_b)+\frac{1}{2}.
  $$

\item If $\Lambda$ is a vertex $\Oo_F$-lattice of type $0$, then

 $$
 n(b,\Lambda)=\ord_{\pi_0}(q(b))-d(\Lambda,\Lambda_b).$$
\end{enumerate}
\end{lemma}
\begin{proof} Claim (2) follows from Claim (1) and (\ref{n(b,Lambda_0)=min}).
 Now we prove Claim (1). 	Without lost of generality we can assume $\ord_{\pi_0}(q(b))=0$. We prove the lemma by induction on $d(\Lambda,\Lambda_b)$. Let us treat the case $d(\Lambda,\Lambda_b)=\frac{1}{2}$ first, i.e. $\Lambda_b \subset \Lambda$. We have by Lemma \ref{lem3.2}
$$
\Lambda_b =\Oo_F w_0 + \Oo_F w_1, \quad  \Lambda = \Oo_F \pi^{-1}w_0 + \Oo_F w_1.
$$
Write $b =xw_0 + y w_1$ with $q(b) = x \bar y + y \bar x \in  \Oo_{F_0}^\times$. Then  $ x,  y \in  \Oo_F^\times$, and  $b \in \Lambda \setminus\pi \Lambda$. Hence $n(b,\Lambda)=0$.

Now we assume the assertion holds for all $\Lambda$ such that $\frac{1}{2}\leq d(\Lambda,\Lambda_b)\leq d+\frac{1}{2}$. For a $\Lambda$ such that $d(\Lambda, \Lambda_b)=d+\frac{1}{2}\neq \frac{1}{2}$, which satisfies the formula $n:=n(b, \Lambda)=-d$, it suffices to show that all its adjacent vertices also satisfy the formula in the lemma.

Choose a basis $\{v_0,v_1\}=\{\pi^{-1}w_0,w_1\}$ of $\Lambda$ with Gram matrix $\pi^{-1} \kzxz {0} {1} {-1} {0}$.  By Lemma \ref{lem3.2}, the $\q+1$ neighbors of $\Lambda$ in the Bruhat-Tits tree are
\begin{align*}
	\Lambda_\infty=\mathrm{span}_{\mathcal{O} _F}\{\pi^{-1}v_0,\pi v_1\},\quad \Lambda_k=\mathrm{span}_{\mathcal{O} _F}\{\pi v_0,\pi^{-1}(kv_0+v_1)\},
\end{align*}
where $k$ runs through the representatives of $\Oo_{F_0}/(\pi_0)$.  We want to show the claim holds for all the neighbors. One of the neighbors will be closer to $\Lambda_b$ than other neighbors. Without loss of generality, we can assume $\Lambda_\infty$ is closer to $\Lambda_b$ than all the other $\Lambda_k$.  By induction, we know \begin{align*}
n(b,\Lambda_\infty)=n+1.
\end{align*}
By the definition of $n(b,\Lambda)$, we can write $$b=\pi^{n}(a_0v_0+a_1v_1)=\pi^{n+1}(\pi^{-1}a_0v_0+\pi^{-1}a_1v_1),$$ where $\ord_{\pi}(a_i)\geq 0$ and $\mathrm{min}\{\ord_{\pi}(a_0),\ord_{\pi}(a_1)\}=0$. Since $n(b,\Lambda_\infty)=n+1$, $\pi^{-1}a_0v_0+\pi^{-1}a_1v_1 \in \Lambda_\infty $ implies $\ord_{\pi}(a_1)\geq 2$. Hence $\ord_{\pi}(a_0)=0$. Claim:
	\begin{align*}\ord_{\pi}(a_1)\geq 2\text{ and }\ord_{\pi}(a_0)=0\implies a_0v_0+a_1v_1 \not \in \Lambda_k.\end{align*} To prove this, assume $a_0v_0+a_1v_1\in \Lambda_k$ which implies that we can find $a_0',a_1'\in \mathcal{O}_F$ such that \begin{align*}a_0v_0+a_1v_1= a_0' \pi v_0+a_1' \pi^{-1}(kv_0+v_1)=(a_0'\pi +a_1' \pi^{-1} k)v_0+a_1'\pi^{-1} v_1.\end{align*}
As a result,
\[\ord_{\pi}(a_0' \pi + a_1'\pi^{-1}k)=\ord_{\pi}(a_0' \pi + a_1 k)\geq
	\mathrm{min}\{\ord_{\pi}(a'_0 \pi),\ord_{\pi}(a_1)\}\geq 1.\]
But then $a_0' \pi + a_1'\pi^{-1}k$ can not equal to $a_0$ since we know $\ord_{\pi}(a_0)=0$,  which is a contradiction.
	
	The claim gives us that $n(b,\Lambda_k)\leq n-1$. we can easily check $\pi(a_0v_0+a_1v_1 )\in \Lambda_k$. Therefore $n(b,\Lambda_k)=n-1$ as claimed.
\end{proof}

Combining Lemma \ref{lem:meaning of n(b,Lambda)} with Lemma \ref{bruhat2}, we have the following:
\begin{corollary}\label{ballcorollary}
	Assume $q(b)\not=0$, $\mathcal{T} (\mathbf{x})$ is a ball centered at the midpoint of $\Lambda_b$ with radius $\ord_{\pi _0}(q(b))+\frac{1}{2}$. In  particular, $\mathcal T(\bx)$ is empty if and only $q(b) \notin \Oo_{F_0}$.
\end{corollary}

\begin{lemma}\label{conelemma}
	Assume that $b\neq 0$ and $q(b)=0$.
	Let $\Lambda$ be a vertex lattice of type $2$ and assume $n(b,\Lambda)=n$. Then there exists a unique type $2$ adjacent lattice $\Lambda_+$ of $\Lambda$ such that $n(b,\Lambda_+)=n+1$. For any other type $2$ adjacent vertex lattice $\Lambda'$ of $\Lambda$, we have $ n(b,\Lambda')=n-1$.
\end{lemma}
\begin{proof}
By scaling $b$ by a power of $\pi$ we can assume that $n=0$. By Lemma  \ref{lem3.2},  the lattice $\Lambda \approx \mathcal{H}$ with Gram matrix $\kzxz {0} {\pi^{-1}} {-\pi^{-1}} {0}$ (as in the introduction).  
So there is some $b_1 \in \Lambda$ such that $(b, b_1) = \pi^{-1}$. Define $b'=b_1 + \pi a b$ where $a=\frac{1}2(b_1, b_1)$. Since $(b_1,b_1)\in \pi^{-1}\Oo_{F_0}$, then $\{b, b'\}$ is a basis of $\Lambda$ with Gram matrix $\pi^{-1} \kzxz {0} {1} {-1} {0}$.
The adjacent type $2$ vertex lattices of $\Lambda$ are
\[\Lambda_\infty=\mathrm{span}_{\Oo_F}\{\pi^{-1}b, \pi b'\}, \ \Lambda_k=\mathrm{span}_{\Oo_F}\{\pi b, \pi^{-1}( k b+ b')\}, \]
where $k$ runs through the representatives of $\Oo_{F_0}/ (\pi_0)$. Then it is obvious that $n(b,\Lambda_\infty)=1$ and $n(b,\Lambda_k)=-1$ for any $k$. The lemma follows.
\end{proof}

\begin{corollary} \label{cor:cone}
   Assume that $b\neq 0$ and $q(b)=0$.  Then $\mathcal{T}(\bx)$ is a cone with the boundary consisting of vertex lattices $\Lambda$ of type $2$  with $n(b,\Lambda)=0$. Starting with such a vertex $\Lambda$,  there is a unique (half) geodesic such that the number $n(b,\Lambda)$ increases along the geodesic. We call such a geodesic an ascending geodesic starting with $\Lambda$. Any two ascending geodesics coincide after finitely many steps. An ascending geodesic can be thought of as an  `axis' of the cone $\mathcal{T}(\bx)$.
\end{corollary}

\subsection{Bruhat-Tits trees of special cycles: rank $2$ case}\label{rank2tree}
The following is an analogue of Lemma 2.11 from \cite{San3}.
\begin{lemma}\label{int between bt tree}
	Let $\bx_1,\bx_2\in \mathbb{V}$, $b_i=\mathbf{x_i}(e)$,\  $n_i= \ord_{\pi_0}(q(b_i))$ and assume $q(b_i)\not=0,$ for $ i\in \{1,2\}$. Assume $n_1\leq n_2$. Suppose that $\mathcal{T}(\bx_1)\cap\mathcal{T}(\bx_2)\neq \emptyset$. Let $\Lambda_{b_1,b_2}$ be the vertex lattice such that
	$$
	d(\Lambda_{b_1},\Lambda_{b_1,b_2})=n_1+\frac{1}{2}-r\hbox{ and } d(\Lambda_{b_1},\Lambda_{b_1,b_2}) + d(\Lambda_{b_2},  \Lambda_{b_1,b_2})=d(\Lambda_{b_1},\Lambda_{b_2}).
	$$
	Here
	$$
	r= \mathrm{min}(\frac{n_1+n_2+1-d(\Lambda_b,\Lambda_{b'})}{2}, n_1+\frac{1}{2}).
	$$
	Then $\mathcal{T}(\bx_1)\cap\mathcal{T}(\bx_2) \subset \mathcal B$ is the ball of radius $r$ centered at either the vertex $\Lambda_{b_1, b_2}$ or the midpoint of the edge $\Lambda_{b_1, b_2}$, depending on whether $r$ is an integer or not.
\end{lemma}
\begin{proof}
It follows from Corollary \ref{ballcorollary} and the fact that $\mathcal{B}$ is a tree.
\end{proof}

\begin{lemma}\label{int of BT tree anti diagonal case}
		For $\bx_1,\bx_2\in \mathbb{V}$, let $b_1=\bx_1(e),b_2=\bx_2(e)$. Assume
	$$
T=\begin{pmatrix}(b_1,b_1)&(b_1,b_2)\\(b_2,b_1)&(b_2,b_2)\end{pmatrix}
 = \begin{pmatrix}0&\pi^{n}\\(-\pi)^{n}&0
 \end{pmatrix},\,  n\geq 0.
$$
Then $\cT(\bx_1)\cap\cT(\bx_2)$ is a ball with center
$\Lambda=\mathrm{span}_{\Oo_F}\{\pi^{-r} b_1, \pi^{-r} b_2\}$  and radius $r$, where $r =[\frac{n+1}2]$ is the integral part of $\frac{n+1}2$.
\end{lemma}
\begin{proof}
		We assume that  $n=2r-1$ ($r\geq 1$) is odd, the case $n$ is even can be proved similarly. Write $b_i= \pi^r v_i$  for $i=1,2$. Define
	\begin{align*}	
		&\Lambda_{k,0}=\mathrm{span}_{\Oo_F}\{\pi^{-k} v_1,\pi^{k+1} v_2\},	&\Lambda_{k,2}=\mathrm{span}_{\Oo_F}\{\pi^{-k} v_1,\pi^{k} v_2\},\ k\in \Z.
    \end{align*}
		Then $\Lambda_{k,0}$  are vertex lattices of type $0$ and $\Lambda_{k,2}$  are vertex lattices of type $2$. It is easy to check that
		\begin{align*}
		&n(b_1,\Lambda_{k,0})=r+k,
		&n(b_1,\Lambda_{k,2})=r+k,\\
		&n(b_2,\Lambda_{k,0})=r-k-1,
		&n(b_2,\Lambda_{k,2})=r-k.
	    \end{align*}
	So $\{\Lambda_{k,0},\Lambda_{k,2}\mid k\geq -r\}$ is the ascending geodesic associated to $b_1$ starting at $\Lambda_{-r,2}$ and $\{\Lambda_{k,0}\mid k\leq r-1\}\cup \{\Lambda_{k,2}\mid k\leq r\}$ is (the inverse of) the ascending geodesic associated to $b_2$ ending at $\Lambda_{r,2}$. Let $\mathcal{C}$ be the intersection of the above two half geodesics, namely, the line segment joining $\Lambda_{-r,2}$ and $\Lambda_{r,2}$. By Lemma \ref{conelemma}, both $n(b_1,\Lambda)$ and $n(b_2,\Lambda)$ are decreasing along a half geodesic starting from any vertex on $\mathcal{C}$ other than $\mathcal{C}$ itself.
	Combine the above facts, it is easy to see that $\cT(\bx_1)\cap \cT(\bx_2)$ is a ball centered at $\Lambda_{0,2}$ with radius $r$.
\end{proof}

 The following lemma is an analogue of   \cite[Lemma 2.15]{San3}. Recall that wo Hermitian matrices $T_1, T_2 \in \hbox{Herm}_2(\Oo_F)$ are said to be equivalent, denoted by $T_1\approx T_2$ if there is a non-singular matrix $g \in \hbox{GL}_2(\Oo_F)$ such that $g^{t} T_1 \bar{g} =T_2$.

\begin{lemma}\label{B-T tree and inner product}
	Let
	\begin{align*}
	T=\begin{pmatrix}(b_1,b_1)&(b_1,b_2)\\(b_2,b_1)&(b_2,b_2)\end{pmatrix}
	\end{align*}
	where $b_i=\bx_i(e)$ for $\bx_i\in \bV$. Assume $q(b_i)\not=0,\ i\in \{1,2\}$, and $T$ is non-singular. Let $n_i=\ord_{\pi_0}(q(b_i))$ and assume $n_1\leq n_2$. Set $d=d(\Lambda_{b_1},\Lambda_{b_2}) $, then
	\begin{enumerate}[label=(\alph*)]
		
		\item $\text{If}~ \mathcal{T}(\mathbf{\bx_1})\subset\mathcal{T}(\mathbf{\bx_2})$ and $ \Lambda_{b_1}\neq\Lambda_{b_2}, \text{then~} T\approx\begin{pmatrix}u_1(-\pi_0)^{\alpha}&0\\0&u_2(-\pi_0)^{\beta}\end{pmatrix}$ where $ u_1,u_2\in \Oo_{F_0}^\times$ and $-u_1u_2\in \mathrm{Nm}(F^{\times})$ with
		\begin{align*}
		\alpha=n_1  ,\quad   \beta=n_2-d.
		\end{align*}
		\item If $\mathcal{T}(\mathbf{\bx_1})\not\subset\mathcal{T}(\mathbf{\bx_2})~,~ and~ \mathcal{T}(\mathbf{\bx_1})\cap\mathcal{T}(\mathbf{\bx_2})\neq\emptyset$,$~then~T\approx\begin{pmatrix}0&\pi^\alpha\\(-\pi)^\alpha&0\end{pmatrix}$ with
		\begin{align*}
		\alpha=n_1+n_2-d.
		\end{align*}
		\item If $\Lambda_{b_1}=\Lambda_{b_2}$, then $T\approx \begin{pmatrix}u_1(-\pi_0)^\alpha&0\\0&u_2(-\pi_0)^\beta\end{pmatrix}$. Here $u_1,u_2$ satisfy the same conditions as in $(a)$ and
		\begin{align*}
		\alpha=n_1, \quad \beta=n_2+2\ord_{\pi}((c_2,c_1')),
		\end{align*}
		where $c_i=\pi^{-n_i}b_i$, and $c_1'\in\Lambda_{b_1}\setminus \pi\Lambda_{b_1}$ such that $(c_1,c_1')=0.$
	\end{enumerate}
\end{lemma}
\begin{proof}
	If $n_1<0$, the lemma holds trivially. So we assume $ n_1 \ge 0$.\newline
	We treat the case $\Lambda_{b_1}\neq \Lambda_{b_2}$ first. Assume $\Lambda\cap\Lambda'=\Lambda_{b_1}$ where $\Lambda$ and $\Lambda'$ are vertex lattices of type 2. We pick a basis $\{v_0,v_1\}$ of $\Lambda$ such that \begin{align*}	\begin{pmatrix}(v_0,v_0)&(v_0,v_1)\\(v_1,v_0)&(v_1,v_1)\end{pmatrix}=\begin{pmatrix}0&\pi^{-1}\\-\pi^{-1}&0\end{pmatrix}.
	\end{align*}
	Without loss of generality, we can assume \begin{align*}\Lambda'=\mathrm{span}_{\mathcal{O}_F}\{\pi^{-1}v_0,\pi v_1\}, \text{ hence } \Lambda_{b_1}=\mathrm{span}_{\mathcal{O}_F}\{v_0,\pi v_1\}.\end{align*}
	By the symmetry of $\mathcal{B}$, we can also assume \begin{align*}\Lambda_{b_2}=\mathrm{span}_{\mathcal{O}_F}\{\pi^{-d}v_0,\pi^{d+1} v_1\}\text{ where }d=d(\Lambda_{b_1},\Lambda_{b_2}).\end{align*}
	By Lemma \ref{bruhat2}, we can write
	\begin{align*}b_1=\pi^{n_1}( \alpha_0v_0+\alpha_1(\pi v_1)),\ b_2=\pi^{n_2}(\alpha_0'(\pi^{-d}v_0)+\alpha_1'(\pi^{d+1}v_1)),\  \alpha_i,  \alpha_i'\in \mathcal{O}_{F}.
	\end{align*}
	Note that $(b_1,b_1)=(-\pi_0)^{n_1}(- \alpha_0\bar{\alpha}_1- \alpha_1\bar{\alpha}_0)$. Since $\ord_{\pi}(- \alpha_0\bar{\alpha}_1- \alpha_1\bar{\alpha}_0)=0$, we conclude that $\ord_{\pi}(\alpha_0)=\ord_{\pi}(\alpha_1)=0$. Similarly, $\ord_{\pi}(\alpha_0')=\ord_{\pi}(\alpha_1')=0$. Then a short computation shows:
	\begin{align*}
	T=\begin{pmatrix}(b_1,b_1)&(b_1,b_2)\\(b_2,b_1)&(b_2,b_2)\end{pmatrix}=\begin{pmatrix}\pi^{2n_1}\cdot(\text{unit})&\pi^{n_1+n_2-d}\cdot(\text{unit})\\(-\pi)^{n_1+n_2-d}\cdot(\text{unit})&\pi^{2n_2}\cdot(\text{unit})
	\end{pmatrix}.
	\end{align*}
	Note that $$\mathcal{T}(\bx_1)\subset\mathcal{T}(\bx_2)\iff n_1\leq n_2-d\iff n_1\leq \frac{n_1+n_2-d}{2}.$$
	
	Now for the proof of $(a)$, observe that if $\mathcal{T}(\bx_1)\subset\mathcal{T}(\bx_2)$ then $2n_1=\mathrm{min}\{\ord_{\pi}(T_{ij})\}$. This implies
	\[T\approx  \begin{pmatrix}u_1(-\pi_0)^{n_1}&0\\0&u_2(-\pi_0)^{n_2-d}\end{pmatrix}.\]
	
	For $(b)$, the assumption implies that $n_1>n_2-d$. Then $n_1+n_2-d=\mathrm{min}\{\ord_{\pi}(T_{ij})\}$, which implies \[T\approx\begin{pmatrix}0&\pi^{n_1+n_2-d}\\(-\pi)^{n_1+n_2-d}&0\end{pmatrix}.\]
	
	The proof for the case $\Lambda_{b_1}=\Lambda_{b_2}$ is essentially the same as the proof in \cite[Lemma 2.15]{San3}, and is left to the reader.\end{proof}
	
\begin{corollary}\label{rank2propersupport}
Let $\bx=(\bx_1,\bx_2)\in \bV^2$ and $b_i=\bx_i(e)$ for $i=1,2$. Assume that $\bx_1$ and $\bx_2$ are linearly independent. Then the naive intersection  $\cZ(\bx)=\cZ(\bx_1)\cap \cZ(\bx_2)$ (resp.  $\cZ^{\Kra}(\bx)$ )  is  supported in finitely many irreducible components of  the special fiber of $\cN_{(1,1)}$ (resp. $\cN^{\Kra}$).
\end{corollary}	
\begin{proof}
This follows easily from Lemma \ref{int between bt tree} and Lemma \ref{int of BT tree anti diagonal case}.
\end{proof}

\section{Decomposition of special divisors in the Kr\"amer model} \label{sect:decomposition}

This section is dedicated to the proof of the following theorem,   which is  slight refinement of Theorem \ref{maintheo2} with $\bx$ replaced by its image $b=\bx(e)$ in $C$.
\begin{theorem}\label{decomposition of special cycle}
	For a vector $\bx\in \bV\setminus\{0\}$, set $b=\bx(e)$ as before. Then  $\cZ^{\Kra}(\bx) =0$ unless $q(b) \in \Oo_{F_0}$.
	\begin{enumerate}
		\item  If $q(b)\not = 0$ and $q(b) \in  \Oo_{F_0}$, then we have the following decomposition of special divisor
		\begin{equation}
		\cZ^{\Kra}(\bx)=\sum_{\Lambda_2\in \mathcal{T}(\bx)} n(b,\Lambda_2) \mathbb{P}_{\bar{\Lambda}_2}+\sum_{\Lambda_0\in \mathcal{T}(\bx)} (n(b,\Lambda_0)+1) \Exc_{\Lambda_0}+\cZ^h(\bx),
		\end{equation}
		where the two summations are over type 2  and type $0$ vertex lattices respectively and $\cZ^h(\bx)\cong \mathrm{Spf}\mathcal{O}_{F}$ is a horizontal divisor meeting the special fiber at $\Exc_{\Lambda_b}$. Recall that $\Lambda_b$ is the unique vertex lattice of type $0$ such that $\pi^{-\ord_{\pi_0}(q(b))}b\in \Lambda_b\setminus \pi\Lambda_b$.
		\item  If $q(b)= 0$, then we have the following decomposition of special divisor
		\begin{equation}
		\cZ^{\Kra}(\bx)=\sum_{\Lambda_2\in \mathcal{T}(\bx)} n(b,\Lambda_2) \mathbb{P}_{\bar{\Lambda}_2}+\sum_{\Lambda_0\in \mathcal{T}(\bx)} (n(b,\Lambda_0)+1) \Exc_{\Lambda_0},
		\end{equation}
		where the two summations are over type 2  and type $0$ vertex lattices respectively.
	\end{enumerate}
\end{theorem}

\subsection{The Horizontal Component}
We say a formal scheme over $\SpfOF$ is horizontal if $\pi$ is not locally nilpotent in its structure sheaf. We say a divisor in $\cN^\Kra$ is irreducible if it is connected and is an irreducible Cartier divisor in every local ring of $\cN^\Kra$.
Let $\cY_s$ be the quasi-canonical lifting of $\bY$ of  level $s$ over $\Oo_{\breve F_0}$  considered by \cite{G} (with $\Oo_s =\Oo_{\breve F_0} + \pi^s \Oo_{\breve F}$ action). In particular, $\cY=\cY_0$ is the canonical lifting. We show that all horizontal cycles in $\cZ^{\Kra}(\bx)$ comes from canonical lifting. 
\begin{theorem}\label{horizonal component of Z(x)}
	Let $\cZ$ be an irreducible horizontal component of $\cZ(\bx)$, then $\cZ\cong \Spf \Oo_{\breve{F}}$. Moreover $\cZ$ intersects with the special fiber of $\cN_{(1, 1)}$ at a superspecial point. 
\end{theorem}

\begin{proof}
By assumption $\cZ=\Spf{R}$ where $R$ is a finite extension of $ \Oo_{\breve{F}_0}$.
Let $X$ be the strict formal $\Oo_{F_0}$-module over $\cZ$ that is the pullback from the universal strict formal $\Oo_{F_0}$-module over $\cN_{(1,1)}$ which by Theorem \ref{alternative description F} carries an $\Oo_B$ action $\iota_{B}$. By the definition of $\cZ(\bx)$, $\bx$ lifts to a homomorphism $\bx:\cY\rightarrow X$. We now define a morphism $\phi:\cY\oplus \cY \rightarrow X$ by
\[\phi(p_1,p_2)=x(p_1)+\iota_{B}(\delta)\circ x(p_2),\]
where $p_1,p_2\in \cY(S)  $ for an $R$-algebra $S$.
We give $ \cY\oplus \cY $ an $\Oo_B$ action $\iota:\Oo_B\rightarrow  \End_{\Oo_{F_0}} (\cY\oplus \cY) \cong \mathrm{M}_{2\times 2}(\Oo_F)$ defined by
\begin{equation}\label{Oo_BactiononY+Y}
\iota(\pi)=\left(\begin{array}{cc}
\pi & 0 \\
0 & -\pi
\end{array}\right) ,\
\iota(\delta)=\left(\begin{array}{cc}
0 & 1 \\
\delta^2 & 0
\end{array}\right).
\end{equation}
Then $\phi$ becomes an $\Oo_B$-linear homomorphism.

We claim that $\phi$ is an isogeny. By \cite[Proposition 1]{tate1967}, the category of connected $\Oo_{F}$ modules and the category of divisible commutative formal Lie groups with $\Oo_{F}$ action are equivalent.
Let $A=R[[T_1,T_2]]$ considered as the structure ring of $\cY\oplus \cY$ and $B$ be the structure ring of $X$. It suffices to show that the induced map $\phi^\sharp:B\rightarrow A$ is injective. If  $I=\mathrm{Ker}(\phi^\sharp)$ is nontrivial, then $\phi $ factors through the sub formal group scheme $X'=\Spf B/I$ of $X$. Since $A$ has characteristic zero, so does $B/I$. By base change to the fraction field $F(R)$ of $R$ we can apply a theorem of Cariter \cite{cartier1962} and conclude that $X'\otimes_R F(R) $ is a one dimensional formal group law over $F(R)$. But by assumption, $X'\otimes_R F(R) $ has an $\Oo_B$ action which is impossible. Hence $I=\{0\}$. This proves the claim.

Our goal is to show that $X$ is in fact isomorphic to $\cY\oplus \cY$.
We define the Tate module for a $\Oo_F$-module $G$ by
\[T(G)=\lim_{\substack{\longleftarrow \\ n}} G[\pi^n].\]
We can identify $T=T(\cY\oplus\cY)$ with $(\Oo_F)^2$ and further with $\Oo_B$ using the $\Oo_B$ action on $T(\cY\oplus\cY)$. To be more specific the element $t=\left(\begin{array}{c}
1  \\
0
\end{array}\right)\in T$ is a generator of $T$ under the $\Oo_B$ action given by \eqref{Oo_BactiononY+Y}. Then the set of  $\Oo_B$-linear isogenies $\phi:\cY\oplus \cY\rightarrow X$ has a one-to-one correspondence with $\Oo_B$ stable lattices $T'$ such that
\begin{equation}
T\subseteq T'\subset T^0 =T\otimes_{\Oo_F} F.
\end{equation}
Since $T$ is a free $\Oo_B$ module of rank $1$, $T'$ must be of the form $\pi^{-n} \cdot T$ for some $n\geq 0$. This shows that $X$ is in fact isomorphic to $\cY\oplus \cY$.
Identifying $X= \cY\oplus \cY$, we see from (\ref{Oo_BactiononY+Y}) that $ j = \hbox{diag}(\pi, \pi) \in \End_{\Oo_B} (X)$. Since $\tr_{F/F_0} \tr j =0$, this implies that  $X \in \mathcal Z(j) \subset \mathcal M$---the special divisor defined in  \cite[Definition 2.1]{KRshimuracurve}. So $\mathcal Z \subset Z(j)$ under the isomorphism $\mathcal N \cong \mathcal M$. Now we conclude from case (iii) of \cite[Proposition 4.5]{KRshimuracurve} that $\cZ \cong \SpfOF$.

By construction we know that the relative Dieudonn\'e module $M(\bar{X})=M(\bY)\oplus M(\bY)$ is fixed by $\tau=\pi V^{-1}$ where $\bar{X}=X\otimes \bar\kay$. By   \cite[Lemma 3.2]{KR3}, $\cZ$ intersects with the special fiber of $\cN_{(1,1)}$ at a superspecial point.
\end{proof}

\subsection{Hodge filtration and equation of special cycles in $\cN_{(1,1)}$}
Now we begin to study the equations of special divisors at a superspecial point. We will use Grothendieck-Messing theory to determine the equations, which in turn requires a description of the Hodge filtrations of the relevant relative Dieudonn\'e crystals. In the following we use the Deligne functor to obtain such a description.

In the rest of this section, for $x=(X,\iota,\lambda,\rho)\in \cN_{(1,1)}(\bar\kay)$, we denote $\rho(M(X))\subset N$ by $M(x)$.
Then identification  $\cN_{(1, 1)}/\bar\kay \cong \cM/\bar\kay$ induces an $\Oo_B$ and thus an $\Oo_E$-action on $M(x)$. This makes $M(x)$ an $\Oo_E \otimes_{\Oo_{F_0}} \Oo_{\breve F_0}$-module and induces  a $\Z/2\Z$ grading (see  \eqref{OoEgrading'})
\[M(x)=M(x)_0\oplus M(x)_1.\]

Now let $x$ be a superspecial point $pt_{\Lambda_0}$ of $\cN_{(1, 1)}$, which implies that $M(x)$ is $\tau$-invariant and that $\Lambda_0=M(x)^\tau$ is a vertex $\Oo_F$-lattice of type $0$ with  $M(x) =\Lambda_0 \otimes_{\Oo_F}\Oo_{{\breve{F}}}$. As explained in Section \ref{section of N_F} we can choose an $\Oo_{F_0}$-basis $\{e_0,f_1\}$ for $M_0(x)^{\tau}$ and $\{f_0,e_1\}$ for $M_1(x)^{\tau}$ such that
\begin{equation}\label{basis of Dieudonne module of superspecial point}
    \pi e_i=f_i,\,  \pi f_i=\pi_0 e_i, \, V e_i=f_i,\,  Vf_i=\pi_0 e_i, \, i=0,\,1.
\end{equation}
and
\begin{equation}\label{standarde_0e_1cont}
    (e_0,e_0)=(e_1,e_1)=0, (e_0,e_1)=-\delta^2.
\end{equation}
Then $\Lambda_0=\mathrm{Span}_{\Oo_F}\{e_0,e_1\}$,  and the vertex lattices of type $2$ containing $\Lambda_0$ are
\begin{align}
&\Lambda_2=\mathrm{Span}_{\Oo_F}\{\pi^{-1}e_0,e_1\},
&\Lambda_2'=\mathrm{Span}_{\Oo_F}\{e_0,\pi^{-1}e_1\}.
\end{align}
Therefore, $x\in \bP_{\bar{\Lambda}_2}\cap \bP_{\bar{\Lambda}_2'}$. Since $\cN\cong \cM$, we have $\cN_{(1, 1)}(\bar\kay)= \cM(\bar\kay)$ and  there should exist homothety classes of rank $2$ $\Oo_{F_0}$-lattices $[\Lambda]$ and $[\Lambda']$ such that $x\in \bP_{[\Lambda]}\cap \bP_{[\Lambda']}$. By \cite[Remark 3.4]{San1}, we can take
\begin{align*}
&\Lambda=M_0(x)^{\tau}=\mathrm{span}_{\Oo_{F_0}}\{e_0,f_1\},\\ &\Lambda'=\pi^{-1}M_1(x)^{\tau}=\mathrm{span}_{\Oo_{F_0}}\{\pi^{-1}f_0,\pi^{-1}e_1\}=\mathrm{span}_{\Oo_{F_0}}\{e_0,\pi_0^{-1}f_1\}.
\end{align*}
Another way to relate these different types of lattices are  the following equations
\begin{equation}
    \Lambda=((\Lambda_2\otimes_{\Oo_{F_0}}\Oo_E)_0)^\tau ,\ \Lambda'=((\Lambda_2'\otimes_{\Oo_{F_0}}\Oo_E)_0)^\tau.
\end{equation}
In particular we can identify $\bP_{[\Lambda]}(\bar\kay)$ with  $\bP_{\bar{\Lambda}_2}(\bar\kay)$ and $\bP_{[\Lambda']}(\bar\kay)$ with  $\bP_{\bar{\Lambda}_2'}(\bar\kay)$.

Consider the Deligne functor $F_{[\Lambda,\Lambda']}$ (see \cite{BC}):
\[ \begin{tikzcd}
\pi_0 \Lambda'  \arrow{r}{} \arrow{d}{\alpha'/\pi_0} & \Lambda  \arrow{d}{\alpha} \arrow{r}{}& \Lambda' \arrow{d}{\alpha'}  \\
\mathcal{L}' \arrow{r}{c'}& \mathcal{L}\arrow{r}{c}&\mathcal{L}'.
\end{tikzcd}
\]

For an $\Oo_{\breve{F}_0}$-algebra $R\in \mathrm{Nilp}_{\Oo_{\breve{F}_0}},$ the conditions of $\alpha$ and $\alpha'$ in Proposition \ref{Deligne functor rep at superspecial point} imply that $\alpha(e_0)$ generates $\mathcal{L}$ and $\alpha'(\pi^{-1}e_1)$ generates $\mathcal{L}'$. We identify $\mathcal L$ and $\mathcal L'$ with $R$ by setting
$$
\alpha(e_0)=1, \quad \hbox { and } \quad  \alpha'(\pi^{-1} e_1)=\alpha' (\pi_0^{-1}f_1)=1.
$$
Let
$$
  t_0= \alpha(f_1) \in \mathcal L = R, \quad \hbox{ and }  \quad         t_1=  \alpha' (e_0)=\alpha'(\pi^{-1} f_0) \in \mathcal L'=R. $$
Then  $c$ is simply the multiplication by $t_0$ and $c'$ is the multiplication by $t_1$.
So by commutativity of the above diagram, we have
\begin{equation}\label{t_0t_1=pi_0}
t_0 t_1=\pi_0.
\end{equation}
Consider \begin{align}\alpha\otimes 1:\Lambda\otimes_{\Oo_{F_0}} R\rightarrow \mathcal{L}\quad \text{and}\quad \alpha'\otimes 1:\Lambda'\otimes_{\Oo_{F_0}} R\rightarrow \mathcal{L} '  \label{map alpha and alpha'}. \end{align}
We have
\begin{align}
&\text{kernel of~} \alpha\otimes 1 =\mathrm{span}_{R}\{f_1\otimes 1 -e_0\otimes t_0\}  \label{ker of alpha'},\\
&\text{kernel of~} \alpha'\otimes 1 =\mathrm{span}_{R}\{\pi^{-1}f_0\otimes 1 -\pi^{-1} e_1\otimes t_1\}.\label{ker of alpha}
\end{align}

For a strict formal $\Oo_{F_0}$-module $X$ over $R$, let $\mathbb{D}(X/R)$ be its relative (to $\Oo_{F_0}$) Dieudonn\'e crystal with  Hodge filtration $\mathrm{Fil} \mathbb{D}(X/R)$, see for example \cite[Section 3]{ACZ}). Then we have the exact sequence:
\begin{equation}
    0\rightarrow \mathrm{Fil} \mathbb{D}(X/R)\rightarrow \mathbb{D}(X/R) \rightarrow \mathrm{Lie} (X/R)\rightarrow 0.
\end{equation}

\begin{proposition}\label{prop of filtration}
Let $x=pt_{\Lambda_0}$ be a superspecial point. For an $\mathcal{O}_{\mathcal{N}_{(1,1)},x}$-algebra $R\in \mathrm{Nilp}_{\Oo_{\breve{F}_0}}$ where $\mathcal{O}_{\mathcal{N}_{(1,1)},x}$ is the local ring of $\cN_{(1,1)}$ at $x$, let $t_0, t_1\in R$ be the image of $T_0, T_1$ under the structure morphism, and $X_{t_0,t_1}$ be the corresponding strict $\Oo_{F_0}$ module over $R$. Then we have the following identifications
	$$
	\mathbb{D}(X_{t_0,t_1}/R)_0=\mathrm{span}_{\Oo_{F_0}}\{e_0,f_1\}\otimes_{\Oo_{F_0}} R, \quad
	\mathrm{Fil} \mathbb{D}(X_{t_0,t_1}/R)_0=\mathrm{span}_{R}\{f_1\otimes 1-e_0\otimes t_0\} ,
$$
and
$$
	\mathbb{D}(X_{t_0,t_1}/R)_1=\mathrm{span}_{\Oo_{F_0}}\{e_1,f_0\}\otimes_{\Oo_{F_0}} R, \quad
	\mathrm{Fil} \mathbb{D}(X_{t_0,t_1}/R)_1=\mathrm{span}_{R}\{f_0\otimes 1-e_1\otimes t_1\}.
$$
\end{proposition}

\begin{proof}
	Since $x$ is a superspecial point, both $0$ and $1$ are critical indices. Hence $M(x)$ is $\tau$ invariant and $M(x)^{\tau}=\mathrm{span}_{\Oo_{F_0}}\{e_0,f_0,e_1,f_1\}$. According to the constructions in \cite{BC}, especially how Deligne's functor is related with the moduli funtor of special formal $\Oo_B$-module, we know that
	\begin{align*}
	\mathbb{D}(X_{t_0,t_1}/R)_0&=M_0(x)^{\tau}\otimes_{\mathcal{O}_{F_0}} R=\mathrm{span}_{\Oo_{F_0}}\{e_0, f_1\}\otimes_{\mathcal{O}_{F_0}} R,
\\
	\mathbb{D}(X_{t_0,t_1}/R)_1&=M_1(x)^{\tau}\otimes_{\mathcal{O}_{F_0}} R=\mathrm{span}_{\Oo_{F_0}}\{e_1, f_0\}\otimes_{\mathcal{O}_{F_0}} R.
\end{align*}
	Under these identifications, the map  $\alpha\otimes 1$ in (\eqref{map alpha and alpha'})  is the natural quotient map from $\mathbb{D}(X_{t_0,t_1}/R)_0$ to $\mathrm{Lie}(X_{t_0,t_1})_0$,  and the maps $\pi\alpha'\otimes 1$  is the natural quotient map from $\mathbb{D}(X_{t_0,t_1}/R)_1$ to $\mathrm{Lie}(X_{t_0,t_1})_1$. Hence $\mathrm{Fil} \mathbb{D}(X_{t_0,t_1}/R)_0$ is the kernel of $\alpha\otimes 1$
	and $\mathrm{Fil} \mathbb{D}(X_{t_0,t_1}/R)_1$ is the kernel of $\pi\alpha'\otimes 1$.  By (\ref{ker of alpha'}) and (\ref{ker of alpha}), we obtain the proposition.
\end{proof}

Similarly for the universal object $\cY$ over  $\mathcal N_{(1,0)}$ and  an $\Oo_{\breve{F}}$-algebra $R  \in \mathrm{Nilp}_{\Oo_{\breve{F}}}$, it is easy to see that
\begin{equation}
\mathbb D(\cY_R) = \hbox{Span}_{R}\{ e \otimes 1, f \otimes 1\}, \quad
\hbox{Fil} \mathbb D(\cY_R) = \hbox{Span}_R \{f\otimes 1 - e \otimes \pi\},
\end{equation}
since  the $\Oo_F$-action coincides with the structure action $\Oo_F \rightarrow R$ on $\hbox{Lie}(\cY_R)$:  $\pi e\otimes 1 = e \otimes \pi$. Here the tensor is over $\Oo_{F_0}$.

\begin{proposition}\label{equation in Pappas model}
Let $x \in \mathcal Z(\bx)(\bar\kay)$ be a superspecial  point $pt_{\Lambda_0}$. Choose a basis  $\{e_0,e_1\}$ of $\Lambda_0$ such that \eqref{basis of Dieudonne module of superspecial point} and \eqref{standarde_0e_1cont} are satisfied. Assume $b =\bx(e)=\alpha_0 e_0 + \alpha_1 e_1 \in \Lambda_0$ where $\alpha_i \in \mathcal{O}_F$. Recall a neighborhood of $x$ in $\cN_{(1,1)}$ is
\begin{equation}\label{breveOmegaequation}
\breve{\Omega}_{[\Lambda,\Lambda'],\breve F}=\mathrm{Spf}(\mathcal{O}_{\breve{F}}[T_0,T_1,(T_0^{\q-1}-1)^{-1},(T_1^{\q-1}-1)^{-1}]/(T_0T_1-\pi_0))^{\wedge}.
\end{equation}
	Then the equations for $\mathcal{Z}(\mathbf{x})$ in the local ring $\Oo_{\mathcal{N}_{(1,1)},x}$ are given by (here $\bar\alpha$ stands for the Galois conjugate of $\alpha$):
	\begin{align}
	\bar{\alpha}_1 T_0-\pi \bar{\alpha}_0=0,\, \hbox{ and  } \bar{\alpha}_0 T_1-\pi \bar{\alpha}_1=0.
	\end{align}
\end{proposition}
\begin{proof}
First,  Lemma \ref{lem3.6} implies $b=\bx(e) \in  \Lambda_0$.
Let
$$
 b'\coloneqq \mathbf{x}(f)=\pi \mathbf{x}(e)=\alpha_0 f_0 + \alpha_1 f_1.
 $$
for some $ \alpha_i\in \Oo_F$. Consider the local ring
	$A\coloneqq\mathcal{O}_{\mathcal{N}_{(1, 1)},x}$.
	Let $I$ denote the ideal corresponding to $\mathcal{Z}(\mathbf{x})$.  Set
	\[A_n\coloneqq A/\pi^{n}A,\quad I_n\coloneqq (I+\pi^{n}A)/\pi^n A  \subset A_n.\]
	Since $A$ is Noetherian, $I$ is $\pi$-adic complete. Hence to prove the proposition, we just need to prove $I_n$ are generated by the images of $\bar{\alpha}_1 T_0-\pi \bar{\alpha}_0$ and $\bar{\alpha}_0 T_1-\pi \bar{\alpha}_1$ in $A_n$ for all $n$. Let $\mathfrak{m}_n$ denote the maximal ideal of $A_n$. Set
	\[B\coloneqq A_n/\mathfrak{m}_nI_n,\quad B'\coloneqq A_n/I_n.\]
	Let $J\coloneqq I_n/\mathfrak{m}_n I_n$, which is the kernel of of the projection $B\rightarrow B'$. Note that $J^2=0$, so it has a $PD$ structure. By Nakayama's lemma, it suffices to show $J$ is generated by images of $\bar{\alpha}_1 T_0-\pi \bar{\alpha}_0$ and $\bar{\alpha}_0 T_1-\pi \bar{\alpha}_1$ in $B$.

   Let $\mathcal X$ be the universal strict formal $\Oo_{F_0}$-module over $\mathcal N_{(1, 1)}$. The natural map
$$
\Oo_{\breve{\Omega}_{[\Lambda,\Lambda'],\breve F}}\rightarrow A \rightarrow A_n \rightarrow B \rightarrow B' \rightarrow \bar\kay
$$
  induces the strict formal $\Oo_{F_0}$-modules  $\mathcal X_B$, $\mathcal X_{B'}$ and $\mathcal X_{\bar\kay}=X$ with $X$ being associated to $x \in \cZ(\bx)(\bar\kay)$.
Since $I$ is the definition ideal of $\cZ(\bx)$, $B'$ is a quotient of  $A/I$,  the quasi-morphism
	\[\rho^{-1}\circ \mathbf{x}:\bY\rightarrow {X}\]
	lifts to a morphism
	\[\mathbf{x}_{{B'}}:\cY_{B'}\rightarrow \mathcal X_{B'} .\]
The associated morphism $\mathbb D(\bx_B')$ lifts to a morphism
$$
\mathbb D(\bx_R): \,  \mathbb{D}(\cY_{R})\rightarrow \mathbb{D}(\mathcal X_{R})
$$
for any ring $R=B/\mathfrak b$ with $\mathfrak b\subset J$, i.e.,  $B\twoheadrightarrow R \twoheadrightarrow B'$ as $J^2=0$. By Grothebdieck-Messing theory,
	$\bx$ lifts to a momorphism
\begin{equation} \label{eq:Lifting}
  \bx_R: \,   \cY_R \rightarrow \mathcal X_R
 \end{equation}
  if and only if
  $$
  \mathbb{D}(\mathbf{x}_{R})(f\otimes 1-e\otimes \pi) =b'\otimes 1-b\otimes \pi \in \mathrm{Fil}\mathbb{D}(\mathcal X_R)
  .$$
Write 	 $\alpha_0=a_0+b_0 \pi,\  \alpha_1=a_1+b_1\pi$. Then
\begin{align*}
&b'\otimes 1-b\otimes \pi\\
&=(\alpha_0 f_0+\alpha_1 f_1)\otimes 1 -(\alpha_0 e_0 + \alpha_1 e_1)\otimes \pi\\
&=((a_0+b_0\pi)f_0+(a_1+b_1\pi)f_1)\otimes 1- ((a_0+b_0\pi)e_0+(a_1+b_1\pi)e_1)\otimes \pi \\
&=(a_0f_0+b_0\pi_0 e_0 +a_1f_1+b_1\pi_0 e_1)\otimes 1-(a_0e_0 + b_0f_0+ a_1e_1+b_1f_1)\otimes \pi \\
&=f_1\otimes (a_1-\pi b_1)-e_0\otimes \pi(a_0-\pi b_0)+f_0\otimes (a_0-\pi b_0)-e_1\otimes \pi(a_1-\pi b_1)\\
&=f_1\otimes \bar{\alpha}_1-e_0\otimes \pi \bar{\alpha}_0+f_0\otimes \bar{\alpha}_0-e_1\otimes \pi \bar{\alpha}_1.
\end{align*}
Combining this with Proposition \ref{prop of filtration}, we see that  the lifting (\ref{eq:Lifting}) exists if and only if
$$
\bar{\alpha}_1 T_0-\pi \bar{\alpha}_0=0,  \hbox{ and  }\bar{\alpha}_0 T_1-\pi \bar{\alpha}_1=0   \hbox{ in R},
$$
i.e.,
$$
\bar{\alpha}_1 T_0-\pi \bar{\alpha}_0, \hbox{ and }  \bar{\alpha}_0 T_1-\pi \bar{\alpha}_1 \in  \mathfrak b.
$$
Here we identify $T_i$ with their images in $R$ via  $A\rightarrow A_n \rightarrow B \rightarrow R$. Since $I$ is the ideal of $\cZ(\bx)$,   the lifting (\ref{eq:Lifting}) exists only when $\mathfrak b=J$. So $J$ is generated by $\bar{\alpha}_1 T_0-\pi \bar{\alpha}_0$ and $  \bar{\alpha}_0 T_1-\pi \bar{\alpha}_1$ as claimed.

\end{proof}

\subsection{Local coordinate charts in the Kr\"amer model}\label{coordinatechartsKramermodel}
Now we describe the local equation of a special divisor in the Kr\"amer model and use it to give a decomposition of special divisor.
Locally around the superspecial point $x \in \cN_{(1, 1)}(\bar\kay)$ corresponding to $ \Lambda_0=\Lambda_2\cap \Lambda_2'$, we have \eqref{breveOmegaequation} and $x$ corresponds to the maximal ideal $\mathfrak{m}_{x}=(T_0,T_1,\pi)$. We need to blow it up to get the exceptional divisor of $\cN^{\Kra}=\cN_{(1, 1)}^{\Kra}$.

 For simplicity, consider (ignoring the other unimportant restrictions of $\breve{\Omega}_{[\Lambda,\Lambda'],F}$)

 \[D\coloneqq \mathrm{Spf}(\Oo_{\breve{F}}[T_0,T_1]/(T_0T_1-\pi_0))^{\wedge}.\]
Let $\mathrm{Bl}_{x} D$ denote the blow-up of $D$ at $\mathfrak m_{x}$, which has three charts. Over the first chart $D_1$, we have
\begin{align} \label{eq:Chart1}
\pi S_{\frac{T_0}{\pi}}=T_0,\ \pi S_{\frac{T_1}{\pi}}=T_1,
\end{align}
where we regard $S_{\frac{x}{\pi}}$ as element in $\mathrm{Frac}(\mathcal{O}_{\breve{F}}[T_0,T_1])$  and
\begin{align}\label{D1 chart of blow up}
D_1&=\mathrm{Spf}\mathcal{O}_{\breve{F}}([T_0,T_1,S_{\frac{T_0}{\pi}},S_{\frac{T_1}{\pi}}]/(\pi S_{\frac{T_0}{\pi}}-T_0, \pi S_{\frac{T_1}{\pi}}-T_1,S_{\frac{T_0}{\pi}}S_{\frac{T_1}{\pi}}-1))^{\wedge} \notag \\
&=\mathrm{Spf}\mathcal{O}_{\breve{F}}([S_{\frac{T_0}{\pi}},S_{\frac{T_1}{\pi}}]/(S_{\frac{T_0}{\pi}}S_{\frac{T_1}{\pi}}-1))^{\wedge}.
\end{align}
Over the second chart $D_2$, we have:
\begin{align} \label{eq:Chart2}
T_0S_{\frac{T_1}{T_0}}=T_1,\,   T_0S_{\frac{\pi}{T_0}}=\pi,
\end{align}
and
\begin{align}\label{D2 chart of blow up}
D_2
&=\mathrm{Spf}\mathcal{O}_{\breve{F}}([T_0,S_{\frac{\pi}{T_0}}]/(T_0S_{\frac{\pi}{T_0}}-\pi))^{\wedge}.
\end{align}
Over the third chart $D_3$, we have:
\begin{align}
T_1S_{\frac{T_0}{T_1}}=T_0,\, T_1S_{\frac{\pi}{T_1}}=\pi,
\end{align}
and
\begin{align}\label{D3 chart of blow up}
D_3
&=\mathrm{Spf}\mathcal{O}_{\breve{F}}([T_1,S_{\frac{\pi}{T_1}}]/(T_1S_{\frac{\pi}{T_1}}-\pi))^{\wedge}
\end{align}
by symmetry. $D_1$, $D_2$ and $D_3$ are glued in the obvious way, and
it is easy to see that $D_1,D_2$ and $D_3$ are all regular.
Let $\Exc$ denote the exceptional divisor of $\cN^{\Kra}$, then
\begin{align}\label{Excequation}
\Exc \cap D_1 &=\mathrm{Spf}\bar{\kay}([S_{\frac{T_0}{\pi}},S_{\frac{T_1}{\pi}}]/(S_{\frac{T_0}{\pi}}S_{\frac{T_1}{\pi}}-1))^{\wedge}, \notag\\
\Exc \cap D_2&=\mathrm{Spf}\mathcal{O}_{\breve{F}}([T_0,S_{\frac{\pi}{T_0}}]/(T_0,T_0S_{\frac{\pi}{T_0}}-\pi))^{\wedge}=\mathrm{Spf}\bar{\kay}([S_{\frac{\pi}{T_0}}])^{\wedge},
\\
\Exc \cap    D_3&=\mathrm{Spf}\mathcal{O}_{\breve{F}}([T_1,S_{\frac{\pi}{T_1}}]/(T_1,T_1S_{\frac{\pi}{T_1}}-\pi))^{\wedge}=\mathrm{Spf}\bar{\kay}([S_{\frac{\pi}{T_1}}])^{\wedge}.
 \notag
\end{align}
$\Exc \cap D_2$ glues with $\Exc \cap D_3$ over $\Exc \cap D_1$ by $S_{\frac{\pi}{T_0}}=\frac{1}{S_{\frac{\pi}{T_1}}}$, so  $\Exc$ is isomorphic to $\bP^1_{\bar{\kay}}$. The projective line $\bP_{\bar{\Lambda}_2}$ only intersects the second chart and is defined by the equation
\begin{equation}\label{PLambda2equation}
    S_{\frac{\pi}{T_0}}=0.
\end{equation}
Similarly $\bP_{\bar{\Lambda}_2'}$ only intersects the third chart and is defined by the equation
\begin{equation}\label{PLambda'2equation}
    S_{\frac{\pi}{T_1}}=0.
\end{equation}
We refer to Example $8.3.53$ of \cite{L} for more details about the blow up considered here.

Now we explain how the coordinates of blow up are related with the moduli problem locally around a superspecial point $x$. Since blow up commutes with flat base change, we have
\begin{align}
\mathcal{N}_{x}^{\Kra}
\coloneqq\mathrm{Bl}_{x} D\times_{\mathcal{N}_{(1,1)}} \Spf\mathcal{O}_{\mathcal{N}_{(1,1)},x}
=\mathrm{Bl}_{x}{\Spf\mathcal{O}_{\mathcal{N}_{(1,1)},x}},
\end{align}
and let	$D_{i, x}$, $i=1, 2,3$ be the three charts for $\cN_{x}^{\Kra}$ coming from $D_i$.

Let $R\in \mathrm{Nilp}_{\Oo_{\breve{F}_0}}$ be an $\mathcal{O}_{D_{1,x}}$-algebra, and $s_{\frac{T_0}{\pi}}, s_{\frac{T_1}{\pi}}\in R$ be the image of $S_{\frac{T_0}{\pi}},S_{\frac{T_1}{\pi}}$ under the structure morphism, and let $t_0$ and $t_1$ be given by (\ref{eq:Chart1}). Then $R$ determines a point $(X_{t_0,t_1}, \mathcal{F})\in D_{1,x}(R)$ where $X_{t_0,t_1}$ is described in Proposition \ref{prop of filtration}, and $\mathcal{F}=\mathrm{Span}_R\{e_0\otimes 1+e_1\otimes s_{\frac{T_1}{\pi}}\} \subset \mathrm{Lie}(X_{t_0,t_1})$ is the filtration of the Kr\"amer moduli problem.

Let $R\in \mathrm{Nilp}_{\Oo_{\breve{F}_0}}$ be an $\mathcal{O}_{D_{2,x}}$-algebra,  and $t_0, s_{\frac{\pi}{T_0}}\in R$ be the image of $T_0,S_{\frac{\pi}{T_0}}$ under the structure morphism, and let $t_1$ be given by (\ref{eq:Chart2}). Then $R$ determines a point $(X_{t_0,t_1}, \mathcal{F})\in D_{2,x}(R)$ where $X_{t_0,t_1}$ is as before and $\mathcal{F}=\mathrm{Span}_R\{e_0\otimes 1+e_1\otimes s_{\frac{\pi}{T_0}}\} \subset \mathrm{Lie}(X_{ t_0,t_1})$. The corresponding description for an $\mathcal{O}_{D_{3,x}}$-algebra is similar.

\subsection{Proof of Theorem \ref{decomposition of special cycle} }
\begin{proof}[Proof of Theorem \ref{decomposition of special cycle}]
In the following, we use $m(\cZ^{\Kra}(\bx) ,\cZ)$ to denote the multiplicity of $\cZ$ in $\cZ^{\Kra}(\bx)$, where $\cZ$ is an irreducible component of $\cZ^{\Kra}(\bx)$. According to Theorem \ref{horizonal component of Z(x)}, the horizontal component of $\cZ(\bx)$ intersects with the special fiber of $\cN_{(1, 1)}$ at some  superspecial points. Moreover, each irreducible component of the special fiber of $\cN_{(1,1)}$ also passes through some superspecial points. Therefore to determine the multiplicity of each irreducible component, it is enough to consider the equations of special divisors at superspecial points and their pullbacks to the Kr\"amer model.

Write $b =\alpha_0 e_0 + \alpha_1 e_1$ as in Proposition \ref{equation in Pappas model}. As before we fix a superspecial point $x=pt_{\Lambda_0}\in \cN_{(1,1)}$ for a vertex lattice $\Lambda_0$ of type $0$.	Recall that the equations of $\mathcal{Z}(\bx)$ in $\mathcal{O}_{\mathcal{N}_{(1,1)},x}$ are:
	\begin{align}
	\bar{\alpha}_1 T_0-\pi \bar{\alpha}_0=0,\quad \bar{\alpha}_0 T_1-\pi \bar{\alpha}_1=0,\ \alpha_i \in \Oo_F.
	\end{align}

	When $\ord_{\pi}(\alpha _0)>\ord_{\pi}(\alpha_1)$, the equations of $\cZ^{\Kra}(\bx)$ in $\mathcal{N}_{x}^{\Kra}$ are

	\begin{align*}
	\left\{ \begin{array}{cc}
	\bar{\alpha}_1\pi (S_{\frac{T_0}{\pi}}-\frac{\bar{\alpha}_0}{\alpha_1})=0, \quad  \bar{\alpha}_1\pi (\frac{\bar{\alpha}_0}{\alpha_1}S_{\frac{T_1}{\pi}}-1)=0  , & \text{in~the~first~chart},  \\
	\bar{\alpha}_1T_0(1-\frac{\bar{\alpha}_0}{\bar{\alpha}_1}S_{\frac{\pi}{T_0}})=0,
\quad \bar{\alpha}_1T_0S_{\frac{\pi}{T_0}}(\frac{\bar{\alpha}_0}{\bar{\alpha}_1}S_{\frac{\pi}{T_0}}-1)=0 ,&   \text{in~the~second~chart},\\ \bar{\alpha}_1T_1S_{\frac{\pi}{T_1}}^2(1-\frac{\bar{\alpha}_0}{\bar{\alpha}_1}\frac{1}{S_{\frac{\pi}{T_1}}})=0,
\quad
\bar{\alpha}_1T_1S_{\frac{\pi}{T_1}} (\frac{\bar{\alpha}_0}{\bar{\alpha}_1}\frac{1}{S_{\frac{\pi}{T_1}}}-1)=0  , &  \text{in~the~third~chart.}
	\end{array}
	\right.
\end{align*}
Notice that $(\frac{\bar{\alpha}_0}{\bar{\alpha}_1}S_{\frac{T_1}{\pi}}-1)$ and $(1-\frac{\bar{\alpha}_0}{\bar{\alpha}_1}S_{\frac{\pi}{T_0}})$ are units in coordinate ring of $D_{1,x}$ and coordinate ring of $D_{2,x}$. In the third chart, we have $T_1S_{\frac{\pi}{T_1}}=\pi$, which implies that
 $$
\frac{\bar{\alpha}_0}{\bar{\alpha}_1}\frac{1}{S_{\frac{\pi}{T_1}}}-1=\alpha T_1 -1
$$
is a unit in $D_{3,x}$ with $ \alpha = \frac{\bar{\alpha}_0}{\bar\alpha_1 \pi} \in \Oo_F$.  So the above equations simplify to
	\begin{align*}
	\left\{ \begin{array}{ccc}
	&\pi^{\ord_{\pi}(\alpha_1)+1} =0,   & \text{in~the~first~chart},  \\
	&\pi^{\ord_{\pi}(\alpha_1)}T_0=T_0^{\ord_{\pi}(\alpha_1)+1}(S_{\frac{\pi}{T_0}})^{\ord_{\pi}(\alpha_1)}=0, &   \text{in~the~second~chart},\\
	&\pi^{\ord_{\pi}(\alpha_1)+1}=T_1^{\ord_{\pi}(\alpha_1)+1}(S_{\frac{\pi}{T_1}})^{\ord_{\pi}(\alpha_1)+1}=0,   &  \text{in~the~third~chart}
	.\end{array}
	\right.\end{align*}
Therefore, it has no horizonal component, and we have by \eqref{Excequation}, \eqref{PLambda2equation} and \eqref{PLambda'2equation}
	\begin{align}
	&m(\cZ^{\Kra}(\bx),\bP_{\bar{\Lambda}_2})=\ord_{\pi}(\alpha_1), \notag \\
	&m(\cZ^{\Kra}(\bx),\bP_{\bar{\Lambda}_2'})=\ord_{\pi}(\alpha_1)+1,\\
	&m(\cZ^{\Kra}(\bx),\Exc_{\Lambda_0})=\ord_{\pi}(\alpha_1)+1. \notag
	\end{align}
Similarly if $\ord_{\pi}(\alpha_0)<\ord_{\pi}(\alpha_1)$, $\cZ^{\Kra}(\bx)$ has  no  horizonal component, and
 \begin{align} &m(\cZ^{\Kra}(\bx),\bP_{\bar{\Lambda}_2})=\ord_{\pi}(\alpha_0)+1, \notag \\
	&m(\cZ^{\Kra}(\bx),\bP_{\bar{\Lambda}_2'})=\ord_{\pi}(\alpha_0),\\
	&m(\cZ^{\Kra}(\bx),\Exc_{\Lambda_0})=\ord_{\pi}(\alpha_0)+1. \notag
	\end{align}
When  $\ord_{\pi}(\alpha_0)=\ord_{\pi}(\alpha_1)$(possible only when  $q(b) \ne 0$), the equations of $\cZ^{\Kra}(\bx)$ in $\mathcal{N}_{x}^{\Kra}$ are
	\begin{align*}
	\left\{ \begin{array}{cc}
	\pi^{\ord_{\pi}(\alpha_0)+1}(S_{\frac{T_1}{\pi}}-\frac{\bar{\alpha}_1}{\bar{\alpha}_0})=0,   & \text{in~the~first~chart},  \\
	T_0^{\ord_{\pi}(\alpha_0)+1}(S_{\frac{\pi}{T_0}})^{\ord_{\pi}(\alpha_0)}(S_{\frac{\pi}{T_0}}-\frac{\bar{\alpha}_1}{\bar{\alpha}_0})=0, &   \text{in~the~second~chart},\\
	T_1^{\ord_{\pi}(\alpha_0)+1}(S_{\frac{\pi}{T_1}})^{\ord_{\pi}(\alpha_0)}(S_{\frac{\pi}{T_1}}-\frac{\bar{\alpha}_0}{\bar{\alpha}_1})=0,   &  \text{in~the~third~chart},
	\end{array}
	\right.
	\end{align*}
	which  implies  by \eqref{Excequation}, \eqref{PLambda2equation} and \eqref{PLambda'2equation}
\begin{align} &m(\cZ^{\Kra}(\bx),\bP_{\bar{\Lambda}_2})=\ord_{\pi}(\alpha_0),\notag \\
	&m(\cZ^{\Kra}(\bx),\bP_{\bar{\Lambda}_2'})=\ord_{\pi}(\alpha_0),\\
	&m(\cZ^{\Kra}(\bx),\Exc_{\Lambda_0})=\ord_{\pi}(\alpha_0)+1. \notag
	\end{align}
In addition,  it has a 	 horizontal component  given by
	\begin{align}\label{equation for horizontal component}
	\left\{ \begin{array}{cc}
	S_{\frac{T_1}{\pi}}-\frac{\bar{\alpha}_1}{\bar{\alpha}_0}=0,  & \text{in~the~first~chart},  \\
	S_{\frac{\pi}{T_0}}-\frac{\bar{\alpha}_1}{\bar{\alpha}_0}=0, &   \text{in~the~second~chart},\\
	S_{\frac{\pi}{T_1}}-\frac{\bar{\alpha}_0}{\bar{\alpha}_1}=0,   &  \text{in~the~third~chart}.
	\end{array}
	\right.
	\end{align}
This is the local equation of $\mathcal Z^h(\bx)$ along the superspecial point $x=pt_{\Lambda_0}$.  In this case, one has to have   $\Lambda_0=\Lambda_b$.
 From the equation, one can also see that the horizontal component is irreducible when  $q(b) \ne 0$, and it intersects with $\Exc_{\Lambda_b}$ at one point (e.g., the image of $S_{\frac{T_1}{\pi}}=\frac{\bar{\alpha}_1}{\bar{\alpha}_0}$ in $\bar\kay$ via  the first chart).
	
	Recall $b=\bx(e)$. Let $n=n(b,\Lambda_0)=\mathrm{min}\{\ord_{\pi}(\alpha_0),\ord_{\pi}(\alpha_1)\}$. Note that
	$$
n(b,\Lambda_2)=
	\left\{\begin{array}{cc}
	n+1,  &  \ord_{\pi}(\alpha_1)>\ord_{\pi}(\alpha_0), \\
	n,     & \text{otherwise,}
	\end{array}
\right.
$$
and
$$
	n(b,\Lambda_2')=
	\left\{\begin{array}{cc}
	n+1,  &  \ord_{\pi}(\alpha_1)<\ord_{\pi}(\alpha_0),  \\
	n,     & \text{otherwise.}
	\end{array}
	\right.
$$
Hence in all three cases we have
	\begin{align}
	&m(\cZ^{\Kra}(\bx),\bP_{\bar{\Lambda}_2})=n(b,\Lambda_2), \notag \\
	&m(\cZ^{\Kra}(\bx),\bP_{\bar{\Lambda}_2'})=n(b,\Lambda_2'),\\
	&m(\cZ^{\Kra}(\bx),\Exc_{\Lambda_0})=n(b,\Lambda_0)+1. \notag
	\end{align}
The above discussion holds for any $\Lambda\in \cT(\bx)$. So we  have
	\begin{equation}
	\cZ^{\Kra}(\bx)=\sum_{\Lambda_2\in \cT(\bx)} n(b,\Lambda_2) \mathbb{P}_{\bar{\Lambda}_2}+\sum_{\Lambda_0\in \cT(\bx)} (n(b,\Lambda_0)+1) \Exc_{\Lambda_0}+\delta_b \cZ^h(\bx),
	\end{equation}
where $\delta_b =0$ or $1$ depends on whether $q(b) =0$ or not. In the latter case,
 $\cZ^h(\bx)\cong \Spf\mathcal{O}_{\breve F}$ is a horizontal divisor meeting the special fiber at $\Exc_{\Lambda_b}$.
\end{proof}

The proof (in particular (\ref{equation for horizontal component})) gives the following corollary, which will be used in next section.

\begin{corollary} \label{cor:Intersection} Let $0\ne \bx \in \bV$ and $b=\bx(e)$.  When  $q(\bx)= 0$,   $\cZ^h(\bx)$ is empty.  When $q(\bx)\ne 0$ with $b \in  \Lambda_b$,  $\cZ^h(\bx) =\Spf \Oo_{\breve F}$ is irreducible,   and intersects with the special fiber of  $\cN^{\Kra}$ at one point on the exceptional divisor $\Exc_{\Lambda_b}$ and is given by image of (\ref{equation for horizontal component}) modulo $\pi$. More precisely, write
$$\pi^{-\ord_{\pi_0} q(b)} b = \alpha_0 e_0 + \alpha_1 e_1,$$
 where $\{e_0,  e_1\}$ is a basis of $\Lambda_b$ given in Proposition \ref{equation in Pappas model}. Then $\alpha_0, \alpha_1 \in \Oo_F^\times$, and the intersection point of $\cZ^h(\bx)$ and $\cN^{\Kra}/\bar\kay$ is  $S_{\frac{T_1}{\pi}}=\frac{\bar{\alpha}
 	_1}{\bar{\alpha}_0} \pmod \pi$ in the first affine chart of the
 	neighborhood of $\cN^{\Kra}_x$ where $x=pt_{\Lambda_b}$.

\end{corollary}

\section{Intersection between special divisors} \label{sect:intersection}
In this section we establish a series of lemmas and  prove Theorems \ref{int between special cycle,diagonal} and \ref{int between special cycle,antidiagonal}, which  are reformulation of Theorem \ref{maintheo1}.

First of all, for two divisors $\cZ_1$ and $\cZ_2$ of $\cN^{\Kra}$ such that $\cZ_1\cap \cZ_2$ is  supported on finitely many irrreducible components of  the special fiber of $\cN^{\Kra}$, we define their intersection number to be
\begin{equation} \label{eq:Intersection}
\cZ_1\cdot \cZ_2 \coloneqq \chi(\cN^{\Kra},\Oo_{\cZ_1}\otimes^{\mathbb{L}}\Oo_{\cZ_2}),
\end{equation}
where $\Oo_{\cZ_i}$ is the structure sheaf of $\cZ_i$, $\otimes^{\mathbb{L}}$ is the derived tensor product on the coherent sheaves on $\cN^{\Kra}$ and $\chi$ is the Euler-Poincar\'e characteristic.

For a full rank lattice $L_{\bx_1,\bx_2}\subset \bV$ with a basis $\{\bx_1, \bx_2\}$, let
\[\mathrm{Int}(L_{\bx_1,\bx_2})=\mathcal Z(\bx_1)\cdot \mathcal Z(\bx_2) .\]
 According to \cite[Corollary D]{Ho2}, this intersection number does not depend on the choice of the basis $\{\bx_1,\bx_2\}$.  First, we recall the following well-known fact.

\begin{proposition}(\cite[Proposition 9.1.21]{L}) \label{Prop Qing Liu}
Assume that $X$ is a regular scheme of dimension $2$, $S$ is a Dedekind scheme of dimension 1 and we have a flat proper morphism $X\rightarrow S$.
Let $s\in S$ be a closed point. The following properties are true:
	\begin{enumerate}[label=(\alph*)]
		\item For any $E\in\mathrm{Div}_s(X)$, we have $E\cdot X_s$=0. Here $\mathrm{Div}_s(X)$ is the set of Cartier divisors of $X$ supported on the special fiber $X_s$.
		\item Let $\Gamma_1,...\Gamma_r$ be the irreducible components of $X_s$ of respective multiplicities $d_1,...,d_r$. Then for any $i\leq r,$ we have
		\begin{align*}
		\Gamma_i^2=-\frac{1}{d_i}\sum_{j\neq i}d_j\Gamma_j\cdot\Gamma_i.
		\end{align*}
	\end{enumerate}
\end{proposition}

\begin{lemma}\label{self int of exc}
Let  $\Lambda_0$ be  a fixed vertex $\Oo_F$-lattice of type $0$. Then
\begin{enumerate}[label=(\alph*)]
	\item $\mathrm{Exc}_{\Lambda_0}\cdot \bP_{\bar\Lambda_2}=\left\{\begin{array}{cc}
			1,   &  \text{if}\ \Lambda_0\subset \Lambda_2,\\
			0,   &  \text{otherwise}.
		\end{array}\right.$
	\item 	$\mathrm{Exc}_{\Lambda_0}\cdot \mathrm{Exc}_{\Lambda_0'}=-2 \delta_{\Lambda_0,  \Lambda_0'}\text{\, for any  type 0 vertex lattice } \Lambda_0'.$
	\item $\mathrm{Exc}_{\Lambda_0}\cdot \cZ^{h}(\bx)=\delta_{\Lambda_0, \Lambda_b} =1$ or $0$ depending on whether $\Lambda_0 =\Lambda_b$ or not.
\end{enumerate}
\end{lemma}
\begin{proof}
By \eqref{Excequation}, \eqref{PLambda2equation} and \eqref{PLambda'2equation}, we can see that
\begin{equation}\label{eq:Exc intersect P}
    \Exc_{\Lambda_0}\cdot \bP_{\bar{\Lambda}_2}=\Exc_{\Lambda_0}\cdot \bP_{\bar{\Lambda}_2'}=1
\end{equation}
if $\Lambda_0=\Lambda_2\cap \Lambda_2'$. For a $\Lambda_2''$ s.t. $\Lambda_0 \not \subset \Lambda_2''$, then clearly $\Exc_{\Lambda_0}\cap \bP_{\bar{\Lambda}_2''}=\emptyset$. Hence $(a)$ is proved. Part $(b)$ follows from Proposition \ref{Prop Qing Liu} and part $(a)$. Part $(c)$ follows from Theorem \ref{horizonal component of Z(x)} and \eqref{equation for horizontal component}.
\end{proof}

\begin{lemma}\label{int of P} Let $\Lambda_2$ be a fixed vertex $\Oo_F$-lattice of type $2$. Then
\begin{enumerate}[label=(\alph*)]
   \item $\mathbb{P}_{\bar{\Lambda}_2}\cdot\mathbb{P}_{\bar{\Lambda}_2'}=\left\{\begin{array}{cc}
		-(q+1),   &  \text{if}\ \Lambda_2=\Lambda_2',\\
		0,   &  \text{otherwise}.	\end{array}\right.$
   \item $\bP_{\bar{\Lambda}_2}\cdot \cZ^{h}(\bx)=0$.
\end{enumerate}
\end{lemma}
\begin{proof}
Note that $\bP_{\bar{\Lambda}_2} \cap \bP_{\bar{\Lambda}_2'}=\emptyset$ if $\Lambda_2 \not = \Lambda_2'$. Moreover, there are $\q+1$ exceptional divisors intersecting with $\mathbb{P}_{\bar{\Lambda}_2}$. Then part $(b)$ of Proposition \ref{Prop Qing Liu} and equation \eqref{eq:Exc intersect P} imply that $\mathbb{P}_{\bar{\Lambda}_2}\cdot\mathbb{P}_{\bar{\Lambda}_2}+(\q+1)=0$. So part $(a)$ follows. It is clear from \eqref{Excequation}, \eqref{PLambda2equation} and \eqref{equation for horizontal component} that $\bP_{\bar\Lambda_2}$ and $\cZ^h(\bx)$ do not intersect. Hence part $(b)$ follows.
\end{proof}
\begin{lemma}\label{int btw P and z(x)}
	$\mathbb{P}_{\bar{\Lambda}_2}\cdot \cZ^{\Kra}(\bx)=\left\{\begin{array}{cc}
	1,   &  \text{if}\ \Lambda_2\in \mathcal{T}(\mathbf{x}),\\
	0,   &  \text{otherwise}.
	\end{array}\right.$
\end{lemma}
\begin{proof}
When  $\Lambda_{2} \not
	\in \mathcal{T}(\bx)$,  the intersection number is obviously 0. When  $\Lambda_2 \in \mathcal{T}(\bx)$.  We have by Theorem \ref{decomposition of special cycle},  Lemma \ref{self int of exc} and \ref{int of P},
\begin{align*}
\mathbb P_{\bar{\Lambda}_2}\cdot \cZ^{\Kra}(\bx)
 =&\bP_{\bar\Lambda_2}\cdot[\sum_{\Lambda'_2\in \cT(\bx)}n(b,\Lambda'_2)\bP_{\bar\Lambda'_2}+\sum_{\Lambda_0\in \cT(\bx)}(n(b,\Lambda_0)+1)\Exc_{\Lambda_0}]\\
 &= -(q+1) n(b, \Lambda_2) + \sum_{\Lambda_0 \subset \Lambda_2} (n(b, \Lambda_0)+1)
 \\
 &=\sum_{\Lambda_0 \subset \Lambda_2} (n(b, \Lambda_0)+ 1 - n(b,  \Lambda_2)).
\end{align*}
Now combining the information from \eqref{n(b,Lambda_0)=min} with Lemmas  \ref{bruhat2} and \ref{conelemma} we see that there is exactly one vertex lattice $\Lambda'$ of type $0$ in $\Lambda_2$ such that $n(b,\Lambda')=n(b,\Lambda_2)$ and for any other vertex lattice $\Lambda$ of type $0$ in $\Lambda_2$ we have $n(b,\Lambda)=n(b,\Lambda_2)-1$. Hence we have
\[\sum_{\Lambda_0 \subset \Lambda_2} (n(b, \Lambda_0)+ 1 - n(b,  \Lambda_2))=1.\]
The lemma follows.
\end{proof}

\begin{lemma}\label{int btw Exc and z(x)}
	$\mathrm{Exc}_{\Lambda_0}\cdot \cZ^{\Kra}(\bx)=\left\{\begin{array}{cc}
	-1,   &  \Lambda_0\in \mathcal{T}(\mathbf{x}),\\
	0,   &  \text{otherwise}.
	\end{array}\right.$
\end{lemma}
\begin{proof}
Assume $\Lambda_2$ and $\Lambda_2'$ are the two vertex lattices of type 2 that contain $\Lambda_0$. We treat the cases $\Lambda_0\neq \Lambda_b$ first.  According to \eqref{n(b,Lambda_0)=min}, Lemma  \ref{bruhat2} and \ref{conelemma}, we can without loss of generality assume that
$$
n(b,\Lambda_2)=n(b,\Lambda_0)+1,  \hbox{  and } n(b,\Lambda_2')=n(b,\Lambda_0).
$$
Then by Theorem \ref{decomposition of special cycle} and Lemma \ref{self int of exc}, we have
	\begin{align*}
	\mathrm{Exc}_{\Lambda_0}\cdot \cZ^{\Kra}(\bx)=&\Exc_{\Lambda_0}\cdot[\sum_{\Lambda_2\in \cT(\bx)}n(b,\Lambda_2)\bP_{\bar\Lambda_2}+\sum_{\Lambda'_0\in \cT(\bx)}(n(b,\Lambda'_0)+1)\Exc_{\Lambda'_0}]\\
	=&n(b,\Lambda_2)\mathrm{Exc}_{\Lambda_0}\cdot\mathbb{P}_{\bar{\Lambda}_2}+n(b,\Lambda_2')\mathrm{Exc}_{\Lambda_0}\cdot\mathbb{P}_{\bar{\Lambda}_2'}+(n(b,\Lambda_0)+1)\mathrm{Exc}_{\Lambda_0}\cdot\mathrm{Exc}_{\Lambda_0}\\
	=&(n(b,\Lambda_0)+1)+n(b,\Lambda_0)-2(n(b,\Lambda_0)+1)\\
	=&-1.
	\end{align*}
Now assume $\Lambda_0=\Lambda_b$, which occurs only when $q(b)\neq 0$. Then notice that  $n(b,\Lambda_2)=n(b,\Lambda_2')=n(b,\Lambda_0)$, and $\Exc_{\Lambda_0}\cdot\cZ^h(\bx)=1$. The rest of the proof is the same as the previous case.
\end{proof}
\begin{lemma}\label{int btw horizontal cycle and z(x)}
Write $b_i =\bx_i(e)$, and assume that $(b_1, b_2) =0$ and $q(b_i) \ne 0$.
	 Then $\cZ^h(\bx_1)\cdot \cZ^h(\bx_2) =0$, and
$$
	\cZ^h(\bx_1)\cdot
	\cZ^{\Kra}(\bx_2)=n(b_2,\Lambda_{b_1})+1.
$$
\end{lemma}
\begin{proof}
Define
$\beta_i=\pi^{-\ord_{\pi_0} q(b_i)} b_i$ for $i=1,2$.
The assumption implies
$$
\Lambda_{b_1}=\Lambda_{b_2} =\Oo_F\beta_1+\Oo_F\beta_2,
$$
and the Gram matrix of $\{\beta_1,\beta_2\}$ is diagonal with diagonal entries in $\Oo^\times_{F_0}$.
Let $\{e_0, e_1\}$ be an $\Oo_F$-basis of $\Lambda_{b_1}$ as in Proposition \ref{equation in Pappas model} and write
 \begin{align*}
	 \beta_1=\alpha_0 e_0+\alpha_1 e_1, \hbox{ and } \beta_2=\alpha'_0 e_0+\alpha_1' e_1.
\end{align*}
By comparing the determinants of the Gram matrices of $\{e_0, e_1\}$ and $\{\beta_1,\beta_2\}$, we know that the transition matrix $\left(\begin{array}{cc}
    \alpha_0 & \alpha_1 \\
    \alpha'_0 & \alpha'_1
\end{array}\right)$ has determinant $\alpha_0  \alpha_1' - \alpha_1 \alpha_0' =u \in \Oo_F^\times$.

 On the other hand, if $\cZ^h(\bx_1)$ and $\cZ^h(\bx_2)$ intersect, the intersection points would be in the special fiber. By Corollary \ref{cor:Intersection}, they intersect at one point, which is also the intersection between them and the exceptional divisor $\Exc_{\Lambda_{b_i}}$.  The same corollary asserts   that the point is given  by  (say in first Chart)
$$
\frac{\bar\alpha_1}{\bar \alpha_0} \equiv  \frac{\bar{\alpha}'_1}{\bar{\alpha}'_0} \pmod \pi.
$$
This implies
$\alpha_0  \alpha_1' - \alpha_1 \alpha_0' \in  \pi \Oo_F$, a contradiction.  So $\cZ^h(\bx_1)\cdot\cZ^h(\bx_2) =0$.
Now we have by Theorem \ref{decomposition of special cycle}, Lemmas \ref{self int of exc} and \ref{int of P},
	\begin{align*}
	\cZ^h(\bx_1)\cdot
	\cZ^{\Kra}(\bx_2)&=\cZ^h(\bx_1)\cdot (n(b_2,\Lambda_{b_1})+1)\mathrm{Exc}_{\Lambda_{b_1}}\\
	&=n(b_2,\Lambda_{b_1})+1.
	\end{align*}
\end{proof}
\begin{theorem}\label{int between special cycle,diagonal}
	Let $\{\bx_1,\bx_2\}$ be a basis of a full rank lattice $L_{\bx_1,\bx_2}\subset \bV$, and $b_1=\bx_1(e),b_2=\bx_2(e)$ as before. Assume
	\begin{align*}
	T= \begin{pmatrix}(b_1,b_1)&(b_1,b_2)\\(b_2,b_1)&(b_2,b_2)\end{pmatrix}\approx \begin{pmatrix}u_1(-\pi_0)^{\alpha}&0\\0&u_2(-\pi_0)^{\beta}\end{pmatrix}
	\end{align*}
	where $ \alpha \leq \beta$, $ u_1,u_2\in \Oo_{F_0}^\times$ and $-u_1u_2\in \mathrm{Nm}(F^{\times})$. When $\alpha \ge 0$, we have
 \[\mathrm{Int}(L_{\bx_1, \bx_2})=\alpha+\beta-\frac{2\q(\q^\alpha-1)}{\q-1}.\]
When  $\alpha <0$, we have  $\mathrm{Int}(L_{\bx_1, \bx_2})=0$.
\end{theorem}

\begin{proof} We may assume $T= \hbox{Diag} (u_1(-\pi_0)^{\alpha}, u_2(-\pi_0)^{\beta})$. In such a case, we see
  that $\mathrm{span}_{\Oo_F}\{\pi^{-\alpha}\bx_1(e),\pi^{-\beta}\bx_2(e)\}=\Lambda_{b_1}=\Lambda_{b_2}$. Moreover $\cT(\bx_1)\subset\cT(\bx_2)$ and $\cT(\bx_1)$ is a ball of radius $\alpha+\frac{1}{2}$ centered at $\Lambda_{b_1}$ by Corollary \ref{ballcorollary}.
  We have by
	Theorem \ref{maintheo2} and Lemmas \ref{int btw P and z(x)},  \ref{int btw Exc and z(x)},   \ref{int btw horizontal cycle and z(x)},  and \ref{bruhat2},
	\begin{align*}
	&\cZ^{\Kra}(\bx_1)\cdot\cZ^{\Kra}(\bx_2)\\
	=&\sum_{\Lambda_2\in \mathcal{T}(\bx_1)}\ n(b_1,\Lambda_2) \mathbb{P}_{\bar{\Lambda}_2}\cdot\cZ^{\Kra}(\bx_2)+\sum_{\Lambda_0\in \mathcal{T}(\bx_1)} (n(b_1,\Lambda_0)+1) \mathrm{Exc}_{\Lambda_0}\cdot\cZ^{\Kra}(\bx_2)\\&+\mathcal{Z}^h(\bx_1)\cdot\cZ^{\Kra}(\bx_2)\\
	=&\sum_{\Lambda_2\in \mathcal{T}(\bx_1)}\ n(b_1,\Lambda_2) -\sum_{\Lambda_0\in \mathcal{T}(\bx_1)} (n(b_1,\Lambda_0)+1)+(n(b_2,\Lambda_{b_1})+1)\\
	=&2(\sum_{i=0}^{\alpha }(\alpha-i)\q^{i})-(\alpha +1+2(\sum_{i=1}^{\alpha }(\alpha+1-i)\q^{i}))+(\beta +1)\\
	=&\alpha+\beta-2\sum_{i=1}^{\alpha}\q^i\\
	=&\alpha+\beta-\frac{2\q(\q^\alpha-1)}{\q-1}.
	\end{align*}
\end{proof}
\begin{remark}\label{rmk about Sankaran's result}
	As we mentioned before, there is a similar result for intersection product between special divisors on a similarly defined RZ space $\cN_E$ obtained by Sankaran in \cite{San3}.  Here $E$ is an unramified quadratic field extension of $F_0$.
\end{remark}

\begin{theorem}\label{int between special cycle,antidiagonal}
	For $\bx_1,\bx_2\in \mathbb{V}$, let $b_1=\bx_1(e),b_2=\bx_2(e)$ as before. Assume
	\begin{align*}
	T=\begin{pmatrix}(b_1,b_1)&(b_1,b_2)\\(b_2,b_1)&(b_2,b_2)\end{pmatrix}\approx \begin{pmatrix}0&\pi^{n}\\(-\pi)^{n}&0\end{pmatrix},
	\end{align*}
	with $n$ odd. Write $r=\frac{n+1}{2}$. Then $\mathrm{Int}(L_{\bx_1, \bx_2}) =0$ unless $n \ge 0$. In such a case,
 \[\Int(L_{\bx_1, \bx_2}) =
 - \frac{(\q+1)(\q^{r}-1)}{\q-1} + 2r.\]
\end{theorem}

\begin{proof} We may assume that $T=\kzxz {0} {\pi^n} {(-\pi)^n} {0}$.  By Lemmas \ref{int btw P and z(x)}, \ref{int btw Exc and z(x)} and Theorem \ref{decomposition of special cycle}, we have
\begin{align*}
    &\cZ^{\Kra}(\bx_1)\cdot \cZ^{\Kra}(\bx_2)\\
    =&\sum_{\Lambda_2\in \mathcal{T}(\bx_1)} n(b_1,\Lambda_2) \mathbb{P}_{\bar{\Lambda}_2}\cdot\cZ^{\Kra}(\bx_2)+\sum_{\Lambda_0\in \mathcal{T}(\bx_1)} (n(b_1,\Lambda_0)+1) \mathrm{Exc}_{\Lambda_0}\cdot\cZ^{\Kra}(\bx_2)\\
    =&\sum_{\Lambda_2 \in  \cT(\bx)} n(b_1, \Lambda_2)
 - \sum_{\Lambda_0 \in \cT(\bx)} (n(b_1, \Lambda_0)+1),
\end{align*}
where $\cT(\bx) =\cT(\bx_1)\cap \cT(\bx_2)$.

 Write $b_i= \pi^r v_i$  for $i=1,2$. Define
\[\Lambda_{k,2}=\mathrm{span}_{\Oo_F}\{\pi^{-k} v_1,\pi^{k} v_2\},\ \Lambda_{k,0}=\mathrm{span}_{\Oo_F}\{\pi^{-k} v_1,\pi^{k+1} v_2\},\ k\in \Z.\]
where $\Lambda_{k,2}$ is a vertex lattice of type $2$ and $\Lambda_{k,0}$ is a vertex lattice of type $0$. It is immediate that
\[n(b_1,\Lambda_{k,0})=n(b_1,\Lambda_{k,2})=r+k.\]
By Lemma \ref{int of BT tree anti diagonal case}, $ \mathcal{T}(\bx)$ is a ball of radius $r$ centered at $\Lambda_{0,2}$. We define $\cC=\cC(\Lambda_{0,2},\Lambda_{r,2})$ to be the geodesic segment joining $\Lambda_{0,2}$ and $\Lambda_{r,2}$.

Now we divide $\cT(\bx)$ into $r+1$ parts $\cL_k$ using $\mathcal C$ as follows: for $0\leq  k\leq r$, let $\cL_k$ be the part of $\cT(\bx)$ such that $\Lambda\in \cL_k$ if and only if $\Lambda_{k,2}$ is the first vertex lattice of type $2$ that the geodesic from $\Lambda$ to $\Lambda_{0,2}$ encounters on $\cC$.  In other words, if we set $\cL_k'$ to be the subtree of $\cT(\bx)\setminus \mathcal C$ that starts from $\Lambda_{k,2}$. Then $\cL_k =\cL_k' \cup \Lambda_{k,0}$ if $\Lambda_{k,0}\in \cT(\bx)$.
In particular, $\cL_r=\{\Lambda_{r,2}\}$.
Then we have
$$
\cZ^{\Kra}(\bx_1)\cdot \cZ^{\Kra}(\bx_2) =\sum_{k=0}^{r} (S(k) -S'(k))
$$
where
$$
S(k)=\sum_{\Lambda_2 \in \cL_k} n(b_1, \Lambda_2),  \hbox{ and } S'(k) =\sum_{\Lambda_0 \in \cL_k} (n(b_1, \Lambda_0)+1)
$$
For $k=0$ we have
\[S(0)=\sum_{i=0}^r (r-i) \q^i,\ S'(0)=\sum_{i=0}^r (r+1-i) \q^i.\]
For $1\leq k \leq r-1$ we have
\[S(k)=r+k+\sum_{i=0}^{r-k-1} (r+k-i-1) (\q-1) \q^i,\]
\[S'(k)=r+k+1+\sum_{i=0}^{r-k-1} (r+k-i) (\q-1) \q^i.\]
For $k=r$ we have
\[S(r)=2r,\ S'(r)=0.\]
Summing these terms up,  we obtain
\begin{align*}
	\cZ^{\Kra}(\bx_1) \cdot \cZ^{\Kra}(\bx_2) =& (S(0)-S'(0))+\sum_{k=1}^{r-1}(S(k)-S'(k))+S(r)\\
	=& -\sum_{i=0}^r \q^i-\sum_{k=1}^{r-1} \q^{r-k}+2r\\
	=& -\frac{\q^{r+1}-1}{\q-1}-\frac{\q(\q^{r-1}-1)}{\q-1}+2r\\
	=&-\frac{(\q+1)(\q^{r}-1)}{\q-1} + 2r.
\end{align*}
as claimed.

\end{proof}

By Theorems \ref{int between special cycle,diagonal} and  \ref{int between special cycle,antidiagonal}, we see that $\hbox{Int}(L_{\bx_1, \bx_2})$ depends only on the Gram matrix $T(b_1, b_2) =-\delta^2  T(\bx_1, \bx_2)$. Since $-\delta^2 \in  \Oo_{F_0}^\times$,  we see that the formulas in both theorems are not affected when we change $T(b_1, b_2)$ to $T(\bx_1, \bx_2)$.  This proves Theorem \ref{maintheo1}.

\section{Local densities and the Kudla-Rapoport Conjecture} \label{sect:localdensity}

In this section we record basic results on  local density and prove Theorem \ref{maintheo3}. Let $L$ be an integral $\Oo_F$-lattice with Gram matrix $S$ (unique up to equivalence), and let $T$ be an $n \times n$ invertible Hermitan matrix over $T$. Recall the local density polynomial $\alpha(L, T,X)= \alpha(S, T, X)$ defined in the introduction. the following explicit formulas are special cases considered in  (\cite{Shi2}).

\begin{theorem}(\cite[Theorem 6.2]{Shi2})\label{localdensity}
	 Let $L$ be an $\Oo_F$ Hermitian lattice with Gram matrix $S=\hbox{diag}(v, 1)$ with $v \in \Oo_{F_0}^\times$, and let $\mathcal H$ be the Hermitian hyperbolic $\Oo_F$-plane defined in the introduction. Set $\epsilon_2 =\chi(-v)$, where $\chi$ be the quadratic character of $F_0^\times$ associated to the extension of $F/F_0$.
	
	\begin{enumerate}
		\item Assume that $T\approx \left(\begin{array}{cc}
			u_1 (-\pi_0)^\alpha & 0 \\
			0 & u_2 (-\pi_0)^\beta
		\end{array}\right)$
with $ \alpha \le \beta$ and $u_i \in \Oo_{F_0}^\times$. Set $\epsilon_1 =\chi(-u_1 u_2)$.   Then
$\alpha(L, T, X)=\alpha(\mathcal H, T, X)=0$ unless $\alpha \ge 0$. Assume $\alpha \ge 0$, we have
		\begin{align*}
			\alpha(L, T, X)
			=& (1-X) (1+ \epsilon_2 + \q \epsilon_2) \sum_{e=0}^\alpha (\q X)^e
			-\epsilon_1 \q^{\alpha+1} X^{\beta+1} (1-X) \sum_{e=0}^{\alpha} (\q^{-1}X)^e
			\\
			-&\epsilon_1 (1+\q) (  X^{\alpha+ \beta +2}+\epsilon_1 \epsilon_2 )
			+ (1+\epsilon_2) \q^{\alpha+1} X^{\alpha+1} (1+\epsilon_1 X^{\beta-\alpha}),
		\end{align*}
and
		
	$$\alpha(\mathcal{H}, T, X)=\left(1-q^{-2} X\right)(\sum_{e=0}^\alpha(q X)^e+\epsilon_1 \sum_{e=\beta+1}^{\alpha+\beta+1} q^{\alpha+\beta+1-e} X^e )$$
In particular,
 $$
  \alpha(L, T)=-\epsilon_1(1+\epsilon_1 \epsilon_2) (1+q) + (1+\epsilon_1) (1+ \epsilon_2) q^{\alpha+1},
  $$
  which is zero  if and only if $\epsilon_1 \epsilon_2 =-1$. Similarly
  $$
  \alpha(\mathcal H, T)= (1+\epsilon_1) (1+q^{-1}) (q^\alpha - q^{-1}),
  $$
which is zero if and only if $\epsilon_1 =-1$.

		\item
		Assume that $T= \left(\begin{array}{cc}
			0 & \pi^n \\
			(-\pi)^n  & 0
		\end{array}\right)$ where $n$ is odd.  Then $\alpha(L, T, X) =\alpha(\mathcal H, T, X)=0$ unless $n \ge -1$.  When $n\ge -1$, we have
		\begin{align*}
			\alpha(L,T,X)=& -\q^{n+2}(1-X)\sum_{e=\frac{n+1}{2}+1}^{n+1}(\q^{-1}X)^e+(1-X) (1+\epsilon_2 + \epsilon_2 q) \sum_{e=0}^{\frac{n+1}{2}} (\q X)^e \\
			&-(\q+1)(\epsilon_2+X^{n+2})+(1+\epsilon_2)(\q+1)\q^{\frac{n+1}{2}}X^{\frac{n+1}{2}+1},
		\end{align*}
		and
		$$\alpha(\mathcal{H}, T, X)=\left(1-q^{-2} X\right)(\sum_{e=0}^{\frac{n+1}{2}}(q X)^e+\sum_{\frac{n+3}{2}}^{n+1} q^{n+1-e} X^e).$$
In particular,
$$
\alpha(L, T)= -(1+\epsilon_2) (q+1) (1- q^{\frac{n+1}2}),
$$
which is zero if and only if $\epsilon_2 =-1$. Finally,
$$
\alpha(\mathcal H, T) =\frac{1+q}{q^2} \left[ -2 + (1+q)q^{\frac{n+1}2}\right] \ne 0.
$$		
	\end{enumerate}
\end{theorem}
\begin{proof}[Proof of Theorem \ref{maintheo3}] First notice that the Gram matrix of $L$ can be chosen to be of the form $S=\hbox{diag}(v, 1)$ with $v \in \Oo_{F_0}^\times$ and that the Gram matrix of $\mathcal H$ is $S'$. Let $T =T(\bx_1,\bx_2)$. As $L$ and $L_{\bx_1, \bx_2}$ represent two different Hermitian spaces of dimension $2$, we have to have
$\epsilon_1 \epsilon_2 =-1$ when $T$ can be diagonalized, and  $\epsilon_2 =-1$ when $T$ is anti-diagonal.  The case that $L$ is isotropic is proved in  \cite{Shi2}. Now we assume $L$ is anisotropic, i.e, $\epsilon_2 =-1$. There are two cases.

{\bf Case 1:} We first assume  that $T\approx \hbox{diag}( u_1 (-\pi_0)^\alpha,  u_2(-\pi_0)^\beta)$ with $ \alpha \le \beta$. In this case $\epsilon_1=1$.  When $\alpha <0$, both sides of the identity are automatically zero.  So we assume $\alpha \ge 0$. Theorem \ref{localdensity} implies that
\begin{align*}
	&\alpha(L,T,X)\\
	=&-\q (1-X) \sum_{e=0}^{\alpha} (\q X)^e- \q^{\alpha+1} X^{\beta+1} (1-X) \sum_{e=0}^{\alpha} (\q^{-1}X)^e-(\q+1) (X^{\alpha+\beta+2}-1).
\end{align*}
So
\begin{align*}
	-\alpha'(L, T)=&\frac{\partial}{\partial X}\alpha(L,T,X)|_{X=1}\\
	=&\q\sum_{e=0}^\alpha \q^e+\q^{\alpha+1}\sum_{e=0}^\alpha \q^{-e}-(\q+1)(\alpha+\beta+2)\\
	=&2\sum_{e=1}^{\alpha+1} q^e-(\q+1)(\alpha+\beta+2).
\end{align*}
Theorem \ref{localdensity}(1) also implies
\begin{equation}  \label{eq:SS}
\alpha(L, S) =2 (1+q), \quad
\end{equation}
and
$$
\alpha(\mathcal H, T) = 2 (1+q^{-1}) (q^\alpha -q^{-1}).
$$
Hence the right hand side of the formula in Theorem \ref{maintheo3} is
\begin{align*}
	&\frac{2}{\alpha(L,S)}[\alpha'(L,T)-\frac{\q^2}{\q^2-1}\alpha(\mathcal H,T)]\\
	=&\frac{1}{\q+1}[-2\sum_{e=1}^{\alpha+1} q^e+(\q+1)(\alpha+\beta+2)+\frac{2}{\q-1}-\frac{2\q^{\alpha+1}}{\q-1}]\\
	=&\alpha+\beta+2-2\sum_{e=0}^\alpha \q^e =\hbox{Int}(L_{\bx_1, \bx_2})
\end{align*}
as claimed by Theorem \ref{maintheo1}.

{\bf Case 2:} Now we treat the anti-diagonal case
\[T \approx \left(\begin{array}{cc}
	0 & \pi^n \\
	(-\pi)^n & 0
\end{array}\right),\  n  \hbox{ odd}.
\]
The case $n <-1$ is trivial as both sides are clearly zero.

When  $n \ge  - 1$, and let $r =\frac{n+1}2$.   we have by Theorem \ref{localdensity}
\begin{align*}
\alpha'(L, T)	&=-\frac{\partial}{\partial X}\alpha(L,T,X)|_{X=1}
\\
	&=1-(\q+1) \sum_{e=0}^{r} \q ^e
	+(\q+1)(n+2).
\end{align*}
Combining this with Theorem \ref{localdensity} (2),  we see that the right hand side of Theorem \ref{maintheo3} is
\begin{align*}
     &\frac{2}{\alpha(L,S)}[\alpha'(L,T)-\frac{\q^2}{\q^2-1}\alpha(\mathcal H,T)]\\
	=&\frac{1}{\q+1}-\sum_{e=0}^r \q^e+2r+1-\frac{\q^r}{\q-1}+\frac{2}{(\q-1)(\q+1)}\\
    =& \hbox{Int}(L_{\bx_1, \bx_2})
\end{align*}
as claimed by Theorem \ref{maintheo3}.

\end{proof}

\bibliographystyle{alpha}
\bibliography{reference}

\begin{thebibliography}{RTW14}

\bibitem[ACZ16]{ACZ}
Tobias Ahsendorf, Chuangxun Cheng, and Thomas Zink.
\newblock $\mathcal{O}$-displays and $\pi$-divisible formal
  $\mathcal{O}$-modules.
\newblock {\em Journal of Algebra}, 457:129--193, 2016.

\bibitem[BC91]{BC}
J-F Boutot and Henri Carayol.
\newblock Uniformisation p-adique des courbes de {S}himura: les
  th{\'e}or{\`e}mes de {{\v C}}erednik et de {D}rinfeld.
\newblock {\em Ast{\'e}risque}, (196-97):45--158, 1991.

\bibitem[BY21]{BY20}
Jan~Hendrik Bruinier and Tonghai Yang.
\newblock Arithmetic degrees of special cycles and derivatives of {S}iegel
  {E}isenstein series.
\newblock {\em Journal of the European Mathematical Society}, 23(5):1613--1674,
  2021.

\bibitem[Car62]{cartier1962}
Pierre Cartier.
\newblock Groupes alg{\'e}briques et groupes formels.
\newblock In {\em Colloq. Th{\'e}orie des Groupes Alg{\'e}briques (Bruxelles,
  1962). Librairie Universitaire, Louvain}, pages 87--111, 1962.

\bibitem[Cho22]{Cho}
Sungyoon Cho.
\newblock Special cycles on unitary shimura varieties with minuscule parahoric
  level structure.
\newblock {\em Mathematische Annalen}, pages 1--67, 2022.

\bibitem[Dri76]{D}
V.~G. Drinfeld.
\newblock Coverings of {$p$}-adic symmetric domains.
\newblock {\em Funkcional. Anal. i Prilo\v{z}en.}, 10(2):29--40, 1976.

\bibitem[DY19]{DY1}
Tuoping Du and Tonghai Yang.
\newblock Arithmetic {S}iegel-{W}eil formula on {$X_0(N)$}.
\newblock {\em Adv. Math.}, 345:702--755, 2019.

\bibitem[Far06]{fargues2006isomorphisme}
Laurent Fargues.
\newblock L'isomorphisme entre les tours de {L}ubin-{T}ate et de {D}rinfeld:
  {D}ecomposition cellulaire de la tour de {L}ubin-{T}ate.
\newblock {\em arXiv preprint math/0603618}, 2006.

\bibitem[Gro86]{G}
Benedict Gross.
\newblock On canonical and quasi-canonical liftings.
\newblock {\em Inventiones mathematicae}, 84(2):321--326, 1986.

\bibitem[GS19]{GS}
Luis~E. Garcia and Siddarth Sankaran.
\newblock Green forms and the arithmetic {S}iegel-{W}eil formula.
\newblock {\em Invent. Math.}, 215(3):863--975, 2019.

\bibitem[How19]{Ho2}
Benjamin Howard.
\newblock Linear invariance of intersections on unitary {R}apoport--{Z}ink
  spaces.
\newblock In {\em Forum Mathematicum}, volume~31, pages 1265--1281, 2019.

\bibitem[HY12]{HYBook}
Benjamin Howard and Tonghai Yang.
\newblock {\em Intersections of {H}irzebruch-{Z}agier divisors and {CM}
  cycles}, volume 2041 of {\em Lecture Notes in Mathematics}.
\newblock Springer, Heidelberg, 2012.

\bibitem[Jac62]{J}
Ronald Jacobowitz.
\newblock Hermitian forms over local fields.
\newblock {\em American Journal of Mathematics}, 84(3):441--465, 1962.

\bibitem[KR00]{KRshimuracurve}
Stephen~S. Kudla and Michael Rapoport.
\newblock Height pairings on {S}himura curves and {$p$}-adic uniformization.
\newblock {\em Invent. Math.}, 142(1):153--223, 2000.

\bibitem[KR14]{KR3}
Stephen Kudla and Michael Rapoport.
\newblock An alternative description of the {D}rinfeld $p$-adic half-plane.
\newblock In {\em Annales de l'Institut Fourier}, volume~64, pages 1203--1228,
  2014.

\bibitem[Kr{\"a}03]{Kr}
Nicole Kr{\"a}mer.
\newblock Local models for ramified unitary groups.
\newblock In {\em Abhandlungen aus dem Mathematischen Seminar der
  Universit{\"a}t Hamburg}, volume~73, pages 67--80. Springer, 2003.

\bibitem[KRY99]{KRYtiny}
Stephen~S. Kudla, Michael Rapoport, and Tonghai Yang.
\newblock On the derivative of an {E}isenstein series of weight one.
\newblock {\em Internat. Math. Res. Notices}, (7):347--385, 1999.

\bibitem[KRY04]{KRYcomp}
Stephen~S. Kudla, Michael Rapoport, and Tonghai Yang.
\newblock Derivatives of {E}isenstein series and {F}altings heights.
\newblock {\em Compos. Math.}, 140(4):887--951, 2004.

\bibitem[KRY06]{KRYbook}
Stephen~S Kudla, Michael Rapoport, and Tonghai Yang.
\newblock {\em Modular Forms and Special Cycles on Shimura Curves.(AM-161)},
  volume 161.
\newblock Princeton university press, 2006.

\bibitem[Kud97]{Kudla97}
Stephen Kudla.
\newblock Central derivatives of {E}isenstein series and height pairings.
\newblock {\em Annals of mathematics}, 146(3):545--646, 1997.

\bibitem[Liu06]{L}
Qing Liu.
\newblock {\em Algebraic geometry and arithmetic curves}.
\newblock Oxford University Press, 2006.

\bibitem[Liu11]{Liu}
Yifeng Liu.
\newblock Arithmetic theta lifting and {$L$}-derivatives for unitary groups,
  {I}.
\newblock {\em Algebra Number Theory}, 5(7):849--921, 2011.

\bibitem[LZ22]{LZ}
Chao Li and Wei Zhang.
\newblock Kudla--{R}apoport cycles and derivatives of local densities.
\newblock {\em Journal of the American Mathematical Society}, 35(3):705--797,
  2022.

\bibitem[Pap00]{P}
Georgios Pappas.
\newblock On the arithmetic moduli schemes of {PEL} {S}himura varieties.
\newblock {\em Journal of Algebraic Geometry}, 9(3):577, 2000.

\bibitem[RSZ18]{RSZ}
Michael Rapoport, Brian Smithling, and Wei Zhang.
\newblock Regular formal moduli spaces and arithmetic transfer conjectures.
\newblock {\em Mathematische Annalen}, 370(3-4):1079--1175, 2018.

\bibitem[RTW14]{RTW}
Michael Rapoport, Ulrich Terstiege, and Sean Wilson.
\newblock The supersingular locus of the shimura variety for
  $\mathrm{GU}(1,n-1)$ over a ramified prime.
\newblock {\em Mathematische Zeitschrift}, 276(3-4):1165--1188, 2014.

\bibitem[RZ96]{RZ}
M~Rapoport and Thomas Zink.
\newblock {\em Period spaces for p-divisible groups}.
\newblock Number 141. Princeton University Press, 1996.

\bibitem[San13]{San1}
Siddarth Sankaran.
\newblock Unitary cycles on {S}himura curves and the {S}himura lift {I}.
\newblock {\em Documenta Mathematica}, 18:1403--1464, 2013.

\bibitem[San17]{San3}
Siddarth Sankaran.
\newblock Improper intersections of {K}udla--{R}apoport divisors and
  {E}isenstein series.
\newblock {\em Journal of the Institute of Mathematics of Jussieu},
  16(5):899--945, 2017.

\bibitem[Shi18]{Shi1}
Yousheng Shi.
\newblock Special cycles on the basic locus of unitary {S}himura varieties at
  ramified primes.
\newblock {\em arXiv preprint arXiv:1811.11227}, 2018.

\bibitem[Shi20]{Shi2}
Yousheng Shi.
\newblock Special cycles on unitary {S}himura curves at ramified primes.
\newblock {\em To appear in Manuscripta Mathematica}, 2020.

\bibitem[Tat67]{tate1967}
John~T Tate.
\newblock p-divisible groups.
\newblock In {\em Proceedings of a conference on Local Fields}, pages 158--183.
  Springer, 1967.

\end{thebibliography}
\end{document}